\documentclass[11pt]{amsart}

\usepackage{amsmath, amscd, amssymb}
\usepackage[frame,cmtip,arrow,matrix,line,graph,curve]{xy}
\usepackage{graphpap, color}
\usepackage[mathscr]{eucal}
\usepackage{cancel}
\usepackage{verbatim}

\numberwithin{equation}{section}

\def\sO{{\mathscr O}}

\newcommand{\A }{\mathbb{A}}
\newcommand{\CC}{\mathbb{C}}

\newcommand{\PP}{\mathbb{P}}
\newcommand{\QQ}{\mathbb{Q}}

\newcommand{\ZZ}{\mathbb{Z}}

\newcommand{\cal}{\mathcal}

\def\cE{{\cal E}}
\def\cF{{\cal F}}

\def\cH{{\cal H}}

\def\cM{{\cal M}}
\def\cN{{\cal N}}

\def\cQ{{\cal Q}}

\def\cU{{\cal U}}

\def\cY{{\cal Y}}
\def\cZ{{\cal Z}}

\def\cX{\mathcal{X} }

\def\fC{\mathfrak{C}}

\def\fE{\mathfrak{E}}

\def\fN{\mathfrak{N}}

\def\loc{\mathrm{loc}}

\let\fE=\cE

\def\and{\quad{\rm and}\quad}
\def\lra{\longrightarrow }
\def\mapright#1{\,\smash{\mathop{\lra}\limits^{#1}}\,}
\def\mapleft#1{\,\smash{\mathop{\longleftarrow}\limits^{#1}}\,}

  \DeclareMathOperator{\Hom}{Hom}

\DeclareMathOperator{\id}{id}

\DeclareMathOperator{\spec}{Spec}

\newtheorem{prop}{Proposition}[section]
\newtheorem{theo}[prop]{Theorem}
\newtheorem{lemm}[prop]{Lemma}
\newtheorem{coro}[prop]{Corollary}
\newtheorem{rema}[prop]{Remark}
\newtheorem{exam}[prop]{Example}

\newtheorem{defi}[prop]{Definition}

\newtheorem{assu}[prop]{Assumption}

\def\beq{\begin{equation}}
\def\eeq{\end{equation}}

\def\dual{^{\vee}}

\def\virt{^{\mathrm{vir}} }

\def\Spec{\mathrm{Spec}\, }
\def\bbT{\mathbb{T} }
\def\bbL{\mathbb{L} }
\def\DM{Deligne-Mumford }

\def\bG{\mathbb{G}}

\def\bbA{\mathbb{A} }

\def\Po{\PP^1}

\def\redd{{\mathrm{red}}}

\def\tE{\widetilde{E} }

\def\tX{\widetilde{X} }
\def\tF{\widetilde{F} }

\def\tX{\widetilde{X} }

\def\tE{\widetilde{E}}

\def\DM{Deligne-Mumford }

\def\bk{\mathbf{k}}
\def\bcV{\mathbb{V}}

\title{Virtual intersection theories}
\date{{2021.6.15}}
\author{Young-Hoon Kiem}
\address{Department of Mathematics and Research Institute
of Mathematics, Seoul National University, Seoul 08826, Korea}
\email{kiem@snu.ac.kr}

\author{Hyeonjun Park}
\address{Department of Mathematics, Seoul National University, Seoul 08826, Korea}
\email{hyeonjun93@snu.ac.kr}

\thanks{Partially supported by Samsung Science and Technology Foundation SSTF-BA1601-01}

\begin{document}
\begin{abstract} 
We construct virtual fundamental classes in all intersection theories including Chow theory, K-theory and algebraic cobordism for quasi-projective \DM stacks with perfect obstruction theories and prove the virtual pullback formula, the virtual torus localization formula and cosection localization principle. 
\end{abstract}
\maketitle

\section{Introduction}\label{S1}

An \emph{intersection theory} in algebraic geometry is about finding the intersection of subschemes modulo an equivalence relation. The prototype for all intersection theories is Bezout's theorem:

\smallskip

\noindent\emph{If $H_1, \cdots, H_n$ are hypersurfaces of degree $d_1, \cdots, d_n$ in the projective space $\PP^n$, their intersection $H_1\cap \cdots \cap H_n$ has at most $\prod_i d_i$ points or contains an algebraic curve.}

\smallskip

A modern presentation of Bezout's theorem involves a \emph{graded abelian group} $A_*(\PP^n)\cong \ZZ^{n+1}$ 
with intersection product
$$A_i(\PP^n)\otimes A_j(\PP^n)\mapright{\times} A_{i+j}(\PP^n\times \PP^n)\mapright{\Delta^*} A_{i+j-n}(\PP^n)$$
where the first map is the \emph{exterior product} $\xi\otimes \eta\mapsto [\xi\times \eta]$ and the second map is the  intersection with the diagonal, or the  \emph{lci pullback}  by the diagonal. The intersection pairing is defined by 
$$\xi\cdot \eta=p_*\Delta^*(\xi\times \eta)\in \ZZ$$
where $p_*:A_*(\PP^n)\to H_*(\spec \bk)=\ZZ$ is the  \emph{projective pushforward} by  $p:\PP^n\to \spec \bk$.

In late 19th century, Schubert extended Bezout's theorem to Grassmannians and solved many enumerative problems about the numbers of lines and planes satisfying constraints. But Schubert's computations were not entirely rigorous and when listing up 23 problems for the 20th century, Hilbert included it as the 15th problem to provide a 
rigorous foundation of Schubert's enumerative calculus \cite{Klei}. 
In late 1970s,  Fulton and MacPherson developed a rigorous intersection  theory for schemes, a main ingredient of which is the refined Gysin pullback obtained by a deformation to the normal cone together with the excision and the $\bbA^1$-homotopy properties of Chow groups \cite{Ful}. 
Their intersection theory was updated by Vistoli \cite{Vist} for \DM stacks and by Kresch for Artin stacks \cite{Kre}.

In order to provide an algebro-geometric theory of Gromov-Witten invariants, the theory of virtual fundamental class was invented in 1995 by Li-Tian \cite{LiTi} and Behrend-Fantechi \cite{BeFa}, based on the Chow intersection theory in \cite{Ful, Vist, Kre}. Their construction was later relativized to the theory of virtual pullback by Manolache \cite{Man} and actually it is a natural generalization of Fulton-MacPherson's refined Gysin pullback after a choice of an embedding of the normal cone (stack) into a vector bundle (stack).  The virtual fundamental class has played a key role in enumerative geometry and was studied intensively during the past two decades. 

As suggested by Kontsevich \cite{Kont} already in 1995, given a \DM stack $X$ with a perfect obstruction theory, the virtual fundamental class $[X]\virt_{C\!H}$ in Chow theory has its K-theoretic twin, called the virtual structure sheaf  $[\sO_X\virt]=:[X]\virt_K$ (cf. \cite[Remark 5.4]{BeFa}, \cite{YLe}) and the virtual Riemann-Roch was proved in \cite{FaGo}. (See also \cite[Theorem 5.8]{KLk}.)
Quite recently, J. Shen in \cite{Shen} constructed a virtual fundamental class $[X]\virt_\Omega$ in the algebraic cobordism theory $\Omega_*$ of Levine-Morel \cite{LeMo} 
for quasi-projective schemes (cf. Theorem \ref{2.2}). 
So it seems natural to ask if there are \emph{any other homology type theories where the virtual fundamental classes are defined with nice expected properties}. 

Despite the huge number of research articles on virtual {invariants} like Gromov-Witten and Donaldson-Thomas invariants, 
there are only a few methods for handling virtual fundamental classes. 
The three most important techniques for virtual fundamental classes are \begin{enumerate}
\item the virtual pullback formula $f^![X]\virt=[Y]\virt$,
\item the virtual torus localization formula $[X]\virt=\imath_*\frac{[X^T]\virt}{e(N\virt)}$ and
\item the cosection localization principle $\imath_*[X]\virt_\loc=[X]\virt$.
\end{enumerate}
The virtual pullback formula for {Chow theory} was proved in \cite[Propostion 5.10]{BeFa} and \cite{KKP} in the special case of lci pullbacks and in \cite{Man} in full generality. It was extended to cosection localized virtual fundamental classes in \cite{CKL} and to the setting of semi-perfect obstruction theory in \cite{Kis}. Recently Qu proved the virtual pullback formula for the virtual structure sheaves in \cite{Qu}. The virtual torus localization formula for {Chow theory} was proved by Graber-Pandharipande \cite{GrPa} and later generalized to the cosection localized virtual fundamental classes in \cite{CKL}. (See also \cite{Kis}.) The cosection localization principle was proved 
in \cite{KLc} for the {Chow theory} class $[X]\virt_{C\!H}$ and later for the K-theory class $[X]\virt_K=[\sO_X\virt]$ in \cite{KLk}. 
So one may ask whether \emph{these three important techniques may be generalized to virtual fundamental classes for any other homology type intersection theories}.

The goal of this paper is threefold: \begin{enumerate}
\item to extract key properties from \cite{Ful} and codify the notion of an intersection theory for schemes and stacks (Definition \ref{2.1});
\item to construct virtual fundamental classes in all intersection theories (Definitions \ref{2.86}, \ref{3.32}, and \ref{3.82});
\item to prove the virtual pullback formula (Theorems \ref{2.85}, \ref{3.34} and \ref{3.90}), the virtual torus localization formula (Theorems \ref{TLSc} and \ref{TLSt}) and the cosection localization principle (Theorems \ref{3.42} and \ref{4.60}) for virtual fundamental classes in all intersection theories.
\end{enumerate}

For quasi-projective schemes, our results generalize all the known virtual fundamental classes and their key properties. 
Unfortunately, it seems extremely difficult to extend an intersection theory for schemes to stacks, with only one successful example of Chow theory. For example, we don't have an algebraic cobordism theory for Artin stacks.
However, even if we are only interested in schemes, we do need to handle intersection theories of cone stacks and vector bundle stacks. To handle cone stacks even in the absence of an intersection theory for stacks which extends a given intersection theory $H_*$ for quasi-projective schemes, we may define a homology type theory (Definition \ref{2.5})
for algebraic stacks $\cX$  by  the inverse limit \[ \cH_d(\cX)=\varprojlim_{t : T\to \cX} H_{d+d(t)}(T)\]
where $t:T\to \cX$ runs through all smooth morphisms from quasi-projective schemes $T$.
This limit theory has the following nice properties:\begin{enumerate}
\item $\cH_*$ extends $H_*$, i.e. for quasi-projective schemes $X$, $\cH_*(X)=H_*(X)$;
\item for global quotients $\cX=[X/G]$, $\cH_*(\cX)$ coincides with the equivariant Chow theory of Edidin-Graham \cite{EdGr} when $H_*=C\!H_*$ is {Chow theory} and with the equivariant algebraic cobordism of Krishna \cite{Kri} and Heller-Malag\'{o}n-L\'{o}pez \cite{HeMa} when $H_*=\Omega_*$ is Levine-Morel's algebraic cobordism theory;
\item $\cH_*$ is the terminal extension of $H_*$ for quasi-projective schemes to algebraic stacks, i.e. if the intersection theory $H_*$ for quasi-projective schemes extends to algebraic stacks, then there is a functorial homomorphism $H_*(\cX)\to \cH_*(\cX)$ for an algebraic stack $\cX$; 
\item for the category of algebraic stacks admitting a good system of approximations by quasi-projective schemes (Definition \ref{3.1}), $\cH_*$ is a weak intersection theory (Theorem \ref{2.10}), i.e. $\cH_*$ satisfies all the axioms for an intersection theory except that the excision sequence is replaced by a weaker condition (Definition \ref{2.70}). 
\end{enumerate} 
It turns out that this limit intersection theory suffices for our purpose of constructing virtual fundamental classes and proving their key properties.

\medskip

The layout of this paper is as follows. In \S2,  we codify the notion of an intersection theory for schemes and stacks and discuss useful properties as well as examples. In \S3, we introduce the limit intersection theory of stacks from an intersection theory of schemes and show that in the category of algebraic stacks which admit good systems of approximations, the limit theory is a weak intersection theory. In \S4, we prove the virtual pullback formula for all intersection theories of schemes and for their limit intersection theories. In \S5, we establish the cosection localization principle for all intersection theories of schemes and for their limit intersection theories. In \S6, we prove the virtual torus localization formula again for all intersection theories of schemes and for their limit intersection theories.

\medskip

All schemes and stacks in this paper are algebraic, quasi-separated, locally of finite type and defined over a field $\bk$ of characteristic zero. {The hypothesis of characteristic zero is used in many places, especially in the theory of algebraic cobordism (Theorem \ref{2.2}) and the blowup sequence (Lemma \ref{4.55}).}

{Following the notations in \cite{Ful}, the algebraic K-theory of coherent sheaves on a stack $X$ will be denoted by $K_0(X)$ and the algebraic K-theory of vector bundles will be denoted by $K^0(X)$.}

We thank Jinhyun Park, Feng Qu, Junliang Shen, Amalendu Krishna and Jeremiah Heller for useful discussions and comments.

\bigskip

\section{Intersection theories of algebraic stacks}\label{s2}
In this section, we introduce the notion of an intersection theory for algebraic stacks and discuss examples.  We also introduce the notion of a weak intersection theory. 

\medskip

%


Let us summarize the Fulton-MacPherson intersection theory in \cite{Ful}. 
For each scheme $X$, we have a graded abelian group, the Chow group $C\!H_*(X)$ of algebraic cycles $\sum_i n_i[\xi_i]$ modulo rational equivalences \cite[\S1.3]{Ful}. These Chow groups have following structures.
\begin{itemize}
\item \underline{Projective pushforward} \cite[\S1.4]{Ful}: For a projective $f:X\to Y$, we have a pushforward map 
$$f_*:C\!H_*(X)\lra C\!H_*(Y),\quad [\xi]\mapsto \mathrm{deg}(f|_{\xi})\cdot[f(\xi)].$$
\item \underline{Smooth pullback}\footnote{In \cite{Ful}, flat pullbacks and proper pushforwards are constructed for Chow groups. However, we will confine ourselves to smooth pullbacks and projective pushforwards in this paper because general intersection theories such as algebraic cobordism may not have flat pullbacks or proper pushforwards.} \cite[\S1.7]{Ful}: For a smooth morphism $f:X\to Y$ of constant relative dimension $e$, we have a pullback map $$f^*:C\!H_*(Y)\lra C\!H_{*+e}(X),\quad [\eta]\mapsto [f^{-1}\eta].$$
\item  \underline{Exterior product} \cite[\S1.10]{Ful}:  For schemes $X$ and $Y$, we have a map 
$$\times:C\!H_*(X)\otimes C\!H_*(Y)\lra C\!H_*(X\times Y),\quad [\xi]\otimes[\eta]\mapsto [\xi\times \eta].$$
\item  \underline{Intersection with divisor} \cite[\S2.3]{Ful}: For a pseudo-divisor $D=(L,s,Z)$ of a scheme $X$, we have a map 
$$D\cdot : C\!H_*(X) \lra C\!H_{*-1}(Z).$$ 
\end{itemize}

These satisfy natural compatibility conditions and the following.
\begin{itemize}
\item \underline{Excision sequence} \cite[\S1.8]{Ful}: For a closed immersion $\imath:Z\to X$ and its complement $\jmath:X- Z\to X$, we have an exact sequence  
$$C\!H_*(Z)\mapright{\imath_*} C\!H_*(X)\mapright{\jmath^*} C\!H_*(X- Z)\lra 0.$$
\item \underline{Extended homotopy} \cite[\S1.9]{Ful}: If $E$ is a vector bundle of rank $r$ on $X$ and $p:V\to X$ is an $E$-torsor, then the smooth pullback $p^*:C\!H_*(X)\to C\!H_{*+r}(V)$ is an isomorphism. 
\end{itemize}
From these, the following important maps follow.
\begin{itemize}
\item  \underline{Specialization homomorphism} \cite[\S5.2]{Ful}: For a closed immersion $X \hookrightarrow Y$, consider the deformation space $M^{\circ}_{X/Y}$ obtained by blowing up $Y\times \PP^1$ along $X\times \{0\}$ and deleting the strict transformation of $Y\times \{0\}$. Then $M^{\circ}_{X/Y}$ is flat over $\PP^1$ and the fiber over $t=0$ (resp. $t\ne 0$) is the normal cone $C_{X/Y}$ of $X$ in $Y$ (resp. $Y$). The commutative diagram
\beq \label{2.16}
\xymatrix{
C\!H_{*+1}(C_{X/Y}) \ar[r]^{\imath_*} \ar[rd]_{c_1(\sO_{C_{X/Y}})=0} & C\!H_{*+1} (M^{\circ}_{X/Y}) \ar[r]^{\jmath^*}\ar[d]^{C_{X/Y}\cdot} \ar[rd]^{Y\cdot} & C\!H_{*+1} (Y\times \A^1) \ar[r] \ar[d]^{\cong}  & 0\\
& C\!H_* (C_{X/Y}) & C\!H_*(Y) \ar[l]^{\mathrm{sp}}}
\eeq
gives us the specialization homomorphism $\mathrm{sp} : C\!H_*(Y) \to C\!H_*(C_{X/Y})$.
\item \underline{Refined Gysin pullback} \cite[\S6.2]{Ful}: Consider a Cartesian square
$$\xymatrix{
X\ar[r]^g\ar[d]_q & Y\ar[d]^p\\
Z\ar[r]^f & W
}$$
with $f : Z \hookrightarrow W$ a regular immersion of codimension $c$. Then  the induced closed immersion $C_{X/Y}\hookrightarrow {q}^*N_{Z/W}$ of the normal cone into the pullback of the normal bundle of $Z$ in $W$ gives us the refined Gysin pullback by the composition
\beq \label{2.17}
f^! : C\!H_*(Y)\mapright{\mathrm{sp}} C\!H_*(C_{X/Y})\mapright{} C\!H_*(q^*N_{Z/W})\mapright{\cong} C\!H_{*-c}(X)
\eeq
where the last map is given by the extended homotopy. When $p=\id_Y$ and 
$q=\id_X$, we will write $f^*=f^!$.
\item \underline{Intersection product} \cite[\S8.3]{Ful}: If $X$ is a smooth scheme of dimension $n$, the diagonal embedding $\Delta:X\to X\times X$ is regular so that it gives us the intersection product
$$C\!H_i(X)\otimes C\!H_j(X)\mapright{\times} C\!H_{i+j}(X\times X)\mapright{\Delta^*} C\!H_{i+j-n}(X).$$
\end{itemize}

From the above summary, it seems reasonable to codify the notion of an intersection theory for algebraic stacks as follows.

\begin{defi}[Admissible category] 
\label{2.71}
A full {sub-2-category} $\bcV$ of the 2-category $\mathbf{St}_\bk$ of all algebraic stacks, quasi-separated and locally of finite type over $\bk$, is called \emph{admissible} if \begin{enumerate}
\item $\emptyset, \spec \bk\in \bcV$; 
\item $X\sqcup Y\in \bcV$ if $X, Y\in \bcV$;
\item $X\times_{Z} Y\in \bcV$ if $X\to Z$ and $Y\to Z$ are morphisms in $\bcV$;
\item $X\in\bcV$ when $X\to Y$ is a quasi-projective morphism and $Y\in \bcV$.
\end{enumerate}
 \end{defi}
For instance, the full subcategory $\mathbf{Sch}_\bk$ of all schemes of finite type over $\bk$ is admissible. Likewise, the subcategory $\mathbf{QSch}_\bk$ of all  quasi-projective schemes over $\bk$ is admissible. 
The subcategory $\mathbf{DM}_\bk$ of all \DM stacks over $\bk$ is also admissible.
By (1) and (4), all admissible categories $\bcV$ of algebraic stacks contain $\mathbf{QSch}_\bk$.

\begin{defi}[Intersection theory for stacks] 
\label{2.1} Let $\bcV$ be an admissible category of algebraic stacks in $\mathbf{St}_\bk$.
An \emph{intersection theory for $\bcV$} consists of
\begin{enumerate}
\item[(i)] a $\ZZ$-graded abelian group $H_*(X)$ for $X\in \bcV$ and $\mathbf{1} \in H_0(\spec{\bk})$;
\item[(ii)] (projective pushforward) a graded homomorphism 
$$f_*:H_*(X)\lra H_*(Y)$$
for a projective morphism $f:X\to Y$;
\item[(iii)] (smooth pullback) a graded homomorphism 
$$f^*:H_*(Y)\lra H_{*+e}(X)$$
for a smooth morphism $f:X\to Y$ of constant relative dimension $e$;
\item[(iv)] (refined Gysin pullback) a graded homomorphism
$$f^! : H_*(Y') \lra H_{*-c}(X')$$
for a Cartesian square
\[\xymatrix{
X'\ar[r]^{f'}\ar[d]_{g'} & Y'\ar[d]^g\\
X\ar[r]^f & Y
}\]
where $f : X \hookrightarrow Y$ is a regular immersion of constant codimension $c$;
\item[(v)] (exterior product) a bilinear graded homomorphism 
$$\times: H_*(X)\otimes H_*(Y)\lra H_*(X\times Y), \quad u\otimes v\mapsto u\times v$$
which is commutative and associative with unit $\mathbf{1}$.
\end{enumerate}

For a smooth $X$, by pulling back $\mathbf{1}$ to $X$ by $p:X\to \spec \bk$, we obtain the \emph{fundamental class} $[X]=p^*\mathbf{1}$ of $X$. 
  
By \emph{(iv)}, when $\imath:D=s^{-1}(0)\hookrightarrow X$ is an effective Cartier divisor for a nonzero section $s$ of a line bundle $L$ on $X$, we have the \emph{intersection product} 
$$\imath^!=D\cdot : H_*(X)\lra H_{*-1}(D)$$ with $D$. 
When $L$ is a line bundle on $X$, the zero section $0:X\to L$ defines the first Chern class homomorphism by 
$$c_1(L)=0^!\circ 0_*:H_*(X)\lra H_{*-1}(X).$$

The above items \emph{(i)-(v)} should satisfy the following conditions:
\begin{enumerate}
\item $H_*(\emptyset)=0$ and  
the projective pushforwards give us an isomorphism
$$H_*(X) \oplus X_*(Y) \mapright{} H_*(X\sqcup Y),\quad  X,Y\in \bcV.$$
\item $(\id_{X})_*=\id_{H_*(X)}$ and for projective  $f:X\to Y$ and $g:Y\to Z$,
$$g_*\circ f_*=(g\circ f)_*.$$
\item $(\id_{X})^*=\id_{H_*(X)}$ and for smooth $f:X\to Y$ and $g:Y\to Z$,
$$f^*\circ g^*=(g\circ f)^*.$$
\item For a Cartesian square
$$\xymatrix{
X\ar[r]^g\ar[d]_q & Y\ar[d]^p\\
Z\ar[r]^f & W
}$$
with $f$ projective and $p$ smooth, we have 
$$p^*\circ f_*=g_*\circ q^*.$$
\item For a diagram of Cartesian squares
$$\xymatrix{
X'' \ar[r]^{f''} \ar[d]_{g''} & Y''\ar[r]\ar[d]^{g'} & S \ar[d]^g\\
X'\ar[r]^{f'}\ar[d]_h & Y' \ar[r] \ar[d] & T \\
X\ar[r]^f & Y
}$$
with $f$ a regular immersion, the following hold:
\begin{enumerate}
\item If $g'$ is projective, then
$$ f^! \circ {g'}_* = {g''}_* \circ f^!.$$
\item If $g'$ is smooth, then
$$f^! \circ {g'}^* = {g''}^*\circ f^!.$$
\item If $g$ is a regular immersion, then
$$g^!\circ f^! = f^! \circ g^!.$$
\item (Excess intersection formula) If $f'$ is also a regular immersion and $E=h^*N_{X/Y}/N_{X'/Y'}$ is the excess normal bundle, then
$$f^!= \eta \circ {f'}^!$$
where $0_E : X'' \to g''^*E$ is the zero section and $\eta=0_E^! \circ {0_E}_*$.
\end{enumerate}
\item For a diagram of Cartesian squares
$$\xymatrix{
X'\ar[r]^h \ar[d] & Y' \ar[r]^k \ar[d] & Z' \ar[d] \\
X\ar[r]^f & Y \ar[r]^g & Z,
}$$
if $f$ and $g$ are regular immersions, then
$$f^! \circ g^! = (g\circ f)^!.$$
\item If a local complete intersection morphism $f:X\to Y$ of constant relative dimension $d$ factors as $f=h\circ g$, with $g:X\to Z$ a regular closed immersion and $h:Z\to Y$ a smooth morphism, then the lci pullback of $f$ defined by
$$f^*:=g^!\circ h^*:H_*(Y) \lra H_{*+d}(X)$$ is independent of the factorization $f=h\circ g$. 
\item For projective morphisms $f$ and $g$,
$$\times \circ (f_* \otimes g_*) = (f\times g)_* \circ \times.$$
\item For smooth morphisms $f$ and $g$,
$$\times \circ (f^* \otimes g^*) = (f\times g)^* \circ \times.$$ 
\item For two Cartesian squares
$$\xymatrix{
X_1\ar[r]^h\ar[d] & Y_1\ar[d]\\
Z_1\ar[r]^f & W_1
}\qquad 
\xymatrix{
X_2\ar[r]^k\ar[d] & Y_2\ar[d]\\
Z_2\ar[r]^g & W_2
}$$
with $f$, $g$ regular immersions, $\times \circ (f^!\otimes g^!)=(f\times g)^!\circ \times.$
\item[(EH)] \underline{Extended homotopy}: For a vector bundle $E$ of rank $r$ over $X$ and an $E$-torsor $p:V\to X$, the pullback 
$p^*:H_*(X)\mapright{} H_{*+r}(V)$ 
is an isomorphism.
\item[(PB)] \underline{Projective bundle formula}: For a vector bundle $E$ of rank $r$ over $X$ and the associated projective bundle $p:\PP E\to X$, we have an isomorphism
\beq\label{2.24}
\bigoplus_{i=0}^{r-1}H_{*-r+1+i}(X)\mapright{} H_* (\PP E), \quad (\xi_i)\mapsto \sum_i c_1(\sO_{\PP E}(1))^i\cdot(p^*\xi_i).
\eeq
\item[(Exc)] \underline{Excision sequence}: For a closed immersion $\imath:Z\hookrightarrow X$ and its complement $\jmath:X-Z\to X$, we have an exact sequence 
\beq\label{2.19}
H_*(Z)\mapright{\imath_*} H_*(X)\mapright{\jmath^*} H_*(X- Z)\lra 0.
\eeq
\item[(DS)] \underline{Detection by smooth schemes}: For any \emph{scheme} $X\in \mathbf{QSch}_\bk$, the projective pushforwards $f_*$ for projective morphisms $f:Y\to X$ from \emph{smooth quasi-projective schemes} $Y$ give us an isomorphism
\beq\label{4.30} \varinjlim_{Y\to X}H_*(Y) \lra H_*(X)\eeq
where the limit is taken over all projective morphisms $Y\to X$ with $Y$ smooth and the transition maps of the limit are given by the projective pushforwards of projective morphisms $Y\to Y'$ over $X$.
\end{enumerate}
\end{defi}

\bigskip

For any intersection theory $H_*$ for $\bcV\subset \mathbf{St}_\bk$, we can derive the following from Definition \ref{2.1}.

\medskip 
\noindent\underline{Chern classes}: For a vector bundle $E$ of rank $r$ over a scheme $X$, the {$i$-th Chern class}
$$c_i(E) : H_*(X) \lra H_{*-i}(X)$$ 
can be defined by the splitting principle (cf. \cite[Remark 4.1.2]{LeMo}). Chern classes commute with projective pushforwards, smooth pullbacks, exterior products, {refined Gysin pullbacks} and other Chern classes. 
The splitting principle gives us  the Whitney sum formula and the self intersection formula $$0_E^! \circ {0_E}_* = c_r(E)$$ where $0_E : X \hookrightarrow E$ is the zero section. 
The items (5) and (7) in Definition \ref{2.1} imply that for any effective  Cartier divisor $\imath:D\hookrightarrow X$ with $L=\sO_X(D)$, we have $$\imath_*\circ\imath^!=\imath_*\circ D\cdot =c_1(L).$$ 
When $\imath:X\to Y$ is a regular closed immersion of constant codimension $m$, by deforming $Y$ to the normal bundle $N_{X/Y}$ of $X$ in $Y$, we have 
\beq\label{6.51}\imath^!\circ \imath_*=c_m(N_{X/Y}).\eeq 
Furthermore, if $p : \PP E \to X$ is the associated projective bundle, then  
{$$\sum_i (-1)^i c_i(p^*E\dual)\cdot c_1(\sO_{\PP E}(1))^{r-i}=0.$$}

\medskip 
\noindent\underline{Specialization homomorphism}: For a \DM type morphism $f:X \to Y$ of algebraic stacks, we can define the \emph{specialization homomorphism}
\beq \label{2.18}
\mathrm{sp}_{X/Y} : H_*(Y) \lra H_*(\fC_{X/Y})
\eeq
where $\fC_{X/Y}$ denotes the intrinsic normal cone of $f$ if $X$, $Y$ and the deformation space $M^\circ_{X/Y}$  lie in $\bcV$. Here a morphism $f$ is of \DM type if  
the diagonal $\Delta_f:X\to X\times_Y X$ is unramified and the intrinsic normal cone $\fC_{X/Y}$ is defined as the cone stack over $X$ associated to the groupoid $[C_{R/S}\rightrightarrows C_{U/V}]$ 
for a commutative diagram
$$\xymatrix{
U\ar[r]\ar[d] & V\ar[d]\\
X\ar[r]^f &Y}$$
where the vertical arrows are smooth {surjective} and the top horizontal arrow $U\to V$ is a closed immersion of schemes
with $R=U\times_XU$ and $S=V\times_YV$ (cf. \cite[Definitions 2.22 \& 2.30]{Man}). 

Indeed, we have the deformation space $M^{\circ}=M^{\circ}_{X/Y}$ which is flat over $\PP^1$ and whose fiber over $0$ is the intrinsic normal cone $C=\fC_{X/Y}$ while $M^\circ-C=Y\times \bbA^1$ by \cite[\S5.1]{Kre}. Let us assume that $X, Y, M^\circ\in \bcV$. 

Applying (Exc) to $C\subset M^\circ$, we obtain a commutative diagram
$$\xymatrix{
H_{*+1}(C) \ar[r]^{\imath_*} \ar[rd]_{c_1(\sO_C)=0} & H_{*+1} (M^{\circ}) \ar[r]^{\jmath^*}\ar[d]_{\imath^!}  & H_{*+1} (Y\times \A^1) \ar[r]  & 0\\
& H_* (C) & H_*(Y) \ar[u]_{p^*}}$$
where $p:Y\times \bbA^1\to Y$ is the projection and $C\cdot$ is the intersection with the divisor $C$. 
The specialization homomorphism is defined by
$$\mathrm{sp}_{X/Y}=\imath^!\circ (\jmath^*)^{-1}\circ p^*.$$

\medskip 
\noindent\underline{Refined Gysin pullbacks by specialization}:  Consider a Cartesian square
$$\xymatrix{
X\ar[r]^g\ar[d]_q & Y\ar[d]^p\\
Z\ar[r]^f & W
}$$
with $f : Z \hookrightarrow W$ a regular immersion of codimension $c$. Then the induced closed immersion $\imath:C_{X/Y}\hookrightarrow q^*N_{Z/W}$ of cones gives us the homomorphism
\beq \label{2.26}
f^! : H_*(Y)\mapright{\mathrm{sp}} H_*(C_{X/Y})\mapright{\imath_*} H_*(q^*N_{Z/W})\mapright{0^!} H_{*-c}(X)
\eeq
where the last map is the Gysin pullback by the zero section or the isomorphism by the extended homotopy. 
We leave it as an exercise to deduce that \eqref{2.26} coincides with the refined Gysin map for $f$ from Definition \ref{2.1}. 

\medskip 
\noindent\underline{Formal group law for schemes}: There is an $F_H(u,v) \in H_*(\Spec \bk)[[u,v]]$ such that for line bundles $L_1, L_2$ over any $X\in \mathbf{QSch}_\bk$, 
\beq\label{2.74} c_1(L_1\otimes L_2)=F_H(c_1(L_1),c_1(L_2))\eeq
by the same proof in \cite[Proposition 5.2.4]{LeMo}. 
As a formal group law of rank one, $F_H(u,v)$ satisfies the properties like
$F_H(u,0)=u$, $F_H(u,v)=F_H(v,u)$ and $F_H(u, F_H(v,w))=F_H(F_H(u,v),w)$. 
As observed in \cite[\S3.1]{LeMo}, there is a power series $g(u)$ with constant term $-1$ such that
\beq\label{4.51} F_H(u, ug(u))=0.\eeq
Note that the map $$c_1:\mathrm{Pic}(X)\lra \Hom(H_*(X), H_{*-1}(X)), \quad L\mapsto c_1(L)$$
is not a group homomorphism unless $F_H(u,v)=u+v$. It is a natural transformation of sets. 

\bigskip

From the above, the following is immediate. 
\begin{lemm}\label{2.27}
Any intersection theory $H_*$ for the category $\mathbf{QSch}_\bk$ of quasi-projective schemes over $\bk$  is an oriented Borel-Moore homology theory in the sense of Levine-Morel \cite[Definition 5.1.3]{LeMo}.
\end{lemm}
\begin{proof} The axiom (CD) in \cite[Definition 5.1.3]{LeMo} follows easily from the excision sequence \eqref{2.19}. All the other axioms are listed above.  \end{proof}

\begin{exam}[Chow groups] As we summarized above, 
by \cite{Ful}, Chow groups $C\!H_*(X)$ together with smooth pullbacks, projective pushforwards and so on form an intersection theory for the category $\mathbf{Sch}_\bk$ of schemes of finite type over $\bk$. By \cite{Vist}, Vistoli's Chow groups with coefficients in $\QQ$ form an intersection theory for the category $\mathbf{DM}_\bk$ of \DM stacks. By \cite{Kre}, Kresch's Chow groups form an intersection theory for algebraic stacks of finite type which admit stratifications by quotient stacks. A special case of \cite{Kre} is the equivariant Chow theory of Edidin-Graham in \cite{EdGr} which is an intersection theory for quotient stacks. 
\end{exam}

\begin{exam}[Algebraic K-theory]\label{2.65}
Let $K_0(X)$ be the Grothendieck group of coherent sheaves on an algebraic stack $X$. Consider the graded group
$$K_0(X)\otimes_{\ZZ}\ZZ[\beta,\beta^{-1}]=\bigoplus_{n\in \ZZ} K_0(X)\cdot \beta^n$$ 
where $\beta$ is a formal variable of degree $1$ and $K_0(X)$ has degree $0$. For a smooth morphism $f:X\to Y$ of constant relative dimension $d$, the smooth pullback is defined by
\[ f^*([E]\cdot \beta^n)= [f^*E]\cdot \beta^{n+d}.\]
The refined Gysin map $f^!$ is defined by the left derived tensor product. 
For a projective morphism $f:X\to Y$, the projective pushforward is 
\[ f_*([E]\cdot \beta^n)=[Rf_*E]\cdot \beta^n.\] 
For $L\in \mathrm{Pic}(X)$, the first Chern class homomorphism is defined by
\[ c_1(L)(\xi)=(\xi -\xi\otimes [L^\vee])\cdot \beta^{-1}\]
and the formal group law is $F_K(u,v)=u+v-\beta uv.$ 
The exterior product is 
\[ ([E]\cdot \beta^n)\times ([E']\cdot \beta^m)=[E\boxtimes E']\cdot \beta^{m+n}.\]
The excision sequence \eqref{2.19} holds for $K_0(X)$ by \cite[Corollary 15.5]{LaMo} and the proof of \cite[Proposition 7]{BoSe}. {Also the axiom (DS) follows from \cite{Dai}.} All the other axioms in Definition \ref{2.1} are easy to check except (EH) and (PB). These last two axioms are proved in \cite{BoSe} for schemes.
Therefore $K_0(-)\otimes_{\ZZ}\ZZ[\beta,\beta^{-1}]$ is an intersection theory for schemes.
It is probably true that (EH) and (PB) hold for $K_0$ of all algebraic stacks but we don't know a proof. 
\end{exam}

{In fact, there are infinitely many intersection theories for quasi-projective
schemes because we can always construct a new theory by twisting a given theory with Todd classes (cf. \cite[7.4.2]{LeMo}). However there is a universal theory.}

\begin{theo}[Algebraic cobordism] \label{2.2} \cite[Theorem 7.1.3]{LeMo}
There is an intersection theory $\Omega_*(X)$ for quasi-projective schemes $X\in \mathbf{QSch}_\bk$, called the \emph{algebraic cobordism}, which is generated by cobordism cycles 
$$[f:Y\lra X]$$ 
where $f$ is a projective morphism from a smooth quasi-projective scheme $Y$, with relations from the double point degenerations in \cite{LePa}.  The algebraic cobordism theory $\Omega_*$ is universal in the sense that for any intersection theory $H_*$ for $\mathbf{QSch}_\bk$, there is a unique homomorphism $$\Omega_*(X)\lra H_*(X)$$ 
that preserves projective pushforwards, smooth pullbacks, refined Gysin pullbacks and exterior products. 
\end{theo}
\begin{proof} By \cite[Theorem 7.1.3]{LeMo}, the algebraic cobordism is universal among oriented Borel-Moore homology theories for $\mathbf{QSch}_\bk$. Since every intersection theory for $\mathbf{QSch}_\bk$ is an oriented Borel-Moore homology by Lemma \ref{2.27} and the algebraic cobordism {satisfies} all the axioms for an intersection theory for $\mathbf{QSch}_\bk$ by \cite{LeMo}, the theorem is a direct consequence of \cite[Theorem 7.1.3]{LeMo}. 
\end{proof}

Sometimes the excision sequence (Exc) in Definition \ref{2.1} may not be available but a weaker condition may hold. So we introduce the following. 
\begin{defi}[Weak intersection theory for stacks]\label{2.70}
A \emph{weak intersection theory} for an admissible category $\bcV\subset \mathbf{St}_\bk$ consists of (i)-(v) in Definition \ref{2.1} satisfying all the axioms except that (Exc) is replaced by 
\begin{enumerate}
\item[(WEx)] \underline{Weak excision}: For a closed immersion $\imath:Z\hookrightarrow X$ and its complement $\jmath:X-Z\to X$ in $\bcV$, the smooth pullback
$$\jmath^*:H_*(X)\lra H_*(X-Z)$$
is surjective. If $\imath$ is regular of codimension $r$ and the normal bundle $N_{Z/X}$ is trivial,
then there is a unique homomorphism $$\lambda_{Z/X}:H_*(X-Z)\lra H_{*-r}(Z)$$ satisfying $\lambda_{Z/X}\circ \jmath^*=\imath^!:H_*(X)\lra H_{*-r}(Z).$
\end{enumerate}
\end{defi}
It is straightforward that (Exc) implies (WEx) and hence an intersection theory is a weak intersection theory. 

\medskip

\noindent \underline{Specialization}: For a weak intersection theory $H_*$ for 
$\bcV$, we still have the specialization homomorphism \eqref{2.18} as the composite
\beq\label{2.80} \mathrm{sp}_{X/Y}=\lambda_{\fC_{X/Y}/M^\circ_{X/Y}}\circ pr_1^*
: H_*(Y)\lra H_{*+1}(Y\times \bbA^1)\lra  H_*(\fC_{X/Y})\eeq
for a \DM type morphism $f:X\to Y$ of algebraic stacks in $\bcV$ with $M^{\circ}_{X/Y}\in \bcV$, by applying
(WEx) to $\fC_{X/Y}\subset M^\circ_{X/Y}$. Here $M^\circ_{X/Y}$ is the deformation space 
from $Y$ to the  relative intrinsic normal cone $\fC_{X/Y}$ of $f$.

The specialization homomorphism commutes with projective pushforwards and smooth pullbacks.

\bigskip

\section{Limit intersection theory}\label{S4}

As we saw above, there are many interesting intersection theories for quasi-projective schemes in $\mathbf{QSch}_\bk$.
But it is extremely difficult to extend an intersection theory for $\mathbf{QSch}_\bk$ to an intersection theory for a larger category of algebraic stacks. 
However, there is a direct way to construct an extension as a \emph{weak} intersection theory for algebraic stacks as we will see in this section.

\subsection{Limit intersection theory}\label{S4.1}
In this subsection, we introduce the notion of limit theory of an intersection theory for quasi-projective schemes and define natural maps. 

\begin{defi}[Limit intersection theory]\label{2.5}
Let $H_*$ be an intersection theory for $\mathbf{QSch}_\bk$. 
To an algebraic stack $\cX\in \mathbf{St}_\bk$, the \emph{limit (weak) intersection theory} assigns  the inverse limit 
\beq \label{2.6} 
\cH_d(\cX)=\varprojlim_{t : T\to \cX} H_{d+d(t)}(T) 
\eeq
for each $d$, where $t$ is a smooth morphism of constant relative dimension $d(t)$ from a quasi-projective scheme $T$.
For a commutative diagram
$$\xymatrix{
T\ar[rr]^f\ar[dr]_t && T'\ar[dl]^{t'}\\
&\cX
}$$
with $t, t'$ smooth and $T, T'\in \mathbf{QSch}_\bk$, 
we use the lci pullback $$f^*:H_{d+d(t')}(T')\lra H_{d+d(t)}(T)$$ to define the limit. Observe that 
$$f:T\mapright{(\id, f)} T\times_\cX T'\mapright{pr_2} T'$$
is the composition of a {regular local immersion} and a smooth morphism.
\end{defi}

The limit theory immediately comes with the following. 

\medskip
\noindent \underline{Projective pushforward}:
For a projective morphism $f:\cX\to \cY$ of algebraic stacks, any smooth morphism $T'\to \cY$ from a quasi-projective scheme $T'$ induces a smooth morphism $T=T'\times_{\cY}\cX\to \cX$ from the quasi-projective scheme $T$ as well as a projective morphism $f_T:T\to T'$. Hence, the projective pushforward $(f_T)_*:H_*(T)\to H_*(T')$ for each $T'$ gives us the projective pushforward
\beq\label{2.7}
f_*:\cH_d(\cX)=\varprojlim_{t : T\to \cX} H_{d+d(t)}(T)\lra \varprojlim_{t' : T'\to \cY} H_{d+d(t')}(T')=\cH_d(\cY).
\eeq

\medskip
\noindent \underline{Smooth pullback}:
For a smooth morphism $f:\cX\to \cY$ of algebraic stacks of relative dimension $e$, any smooth morphism $t:T\to \cX$ with $T\in \mathbf{QSch}_\bk$ induces a smooth morphism $t'=f\circ t:T'=T\to \cY$ and hence we have the smooth pullback
\beq \label{2.8} 
f^*:\cH_d(\cY)=\varprojlim_{t' : T'\to \cY} H_{d+d(t')}(T')\lra \varprojlim_{t : T\to \cX} H_{d+d(t)+e}(T)=\cH_{d+e}(\cX).
\eeq 

\medskip
\noindent \underline{First Chern class}: 
For a line bundle $L$ over an algebraic stack $\cX$ and a smooth morphism $t:T\to \cX$ from a quasi-projective scheme $T$, we have the first Chern class $c_1(t^*L):H_d(T)\to H_{d-1}(T)$. This gives us the first Chern class homomorphism 
\beq\label{2.9}
c_1(L)=\varprojlim_{t:T\to \cX} c_1(t^*L): \cH_d(\cX)\lra \cH_{d-1}(\cX).
\eeq

\medskip
\noindent \underline{Extension of $H_*$}: 
When $X$ is a quasi-projective scheme, we have a canonical isomorphism
$$H_d(X)\cong \varprojlim_{t:T\to X} H_{d+d(t)}(T)=\cH_d(X)$$
because every smooth morphism $t:T\to X$ factors through $\id_X$.

\medskip
\noindent \underline{Terminal extension}: If $H_*$ is actually an intersection theory for an admissible category $\bcV\subset \mathbf{St}_\bk$, then there is a unique homomorphism
\beq\label{2.75} H_d(\cX)\lra \varprojlim_{t : T\to \cX} H_{d+d(t)}(T) =\cH_d(\cX)\quad \text{for }\cX\in \bcV\eeq
that preserves projective pushforwards, smooth pullbacks and {the first Chern classes.} Indeed, for any smooth $t:T\to \cX$, the smooth pullbacks $t^*:H_*(\cX)\to H_{*+d(t)}(T)$ define the desired homomorphism \eqref{2.75}.

\medskip
\noindent \underline{Compatibility}: It is straightforward that the items (1)-(4) in Definition \ref{2.1} hold for $\cH_*$.

\medskip
\noindent \underline{(EH), (PB) and (DS)}: The condition (DS) in Definition \ref{2.1} holds trivially for $\cH_*$.  (EH) and (PB) hold by Lemma \ref{2.11} below.

\begin{lemm}\label{2.11}
Let $f:\cX\to \cY$ be a smooth quasi-projective morphism of algebraic stacks. For each smooth morphism $t:T\to \cY$ from a quasi-projective scheme $T$, the induced morphism $T\times_\cY\cX\to \cX$ is a smooth morphism from a quasi-projective scheme $T\times_{\cY}\cX$. Then the natural map
$$\cH_d(\cX)=\varprojlim_{t':T'\to \cX} H_{d+d(t')}(T')\lra \varprojlim_{t:T\to \cY} H_{d+d(t)}(T\times_\cY\cX)$$
is an isomorphism for all $d$.
\end{lemm}

\begin{proof}
We construct an inverse map. Fix 
$$(\xi_t)_{\{t:T\to \cY\}}\in \varprojlim_{t:T\to \cY} H_{d+d(t)}(T\times_\cY\cX).$$
For any smooth $t':T'\to \cX$ with $T'\in \mathbf{QSch}_\bk$, the composition $t=f\circ t':T=T'\to \cY$ is also smooth and hence we have an lci morphism $(\id,t'):T'\to T\times_\cY\cX$ over $\cX$. Thus we obtain $\xi_{t'}\in H_{d+d(t')}(T')$ by pulling back $\xi_t$. It is easy to see that 
$$(\xi_{t'})_{\{t':T'\to \cX\}}\in \varprojlim_{t':T'\to \cX}H_{d+d(t')}(T')$$
and the assignment $(\xi_t)\mapsto (\xi_{t'})$ is the inverse map to the natural map in the lemma. 
\end{proof}
\begin{proof}[Proof of (EH) and (PB)]
For the extended homotopy (EH), note that for any smooth $T\to \cX$, $T\times_\cX V\to T$ is a torsor over a vector bundle of rank $r$ over $T$.
Since the extended homotopy $H_*(T)\cong H_{*+r}(T\times_\cX V)$ holds for schemes by assumption, we have an isomorphism
$$\cH_d(V)\cong \varprojlim_{t:T\to \cX} H_{d+d(t)}(T\times_\cX V)\mapleft{p^*}  \varprojlim_{t:T\to \cX} H_{d+d(t)-r}(T)\cong \cH_{d-r}(\cX).$$

The proof of the projective bundle formula (PB) is similar. 
\end{proof}

The limit theory $\cH_*$ in Definition \ref{2.5} arising from an intersection theory $H_*$ for $\mathbf{QSch}_\bk$ has the following additional structures.

\medskip
\noindent \underline{Intersection ring}: 
When $\cX$ is a smooth algebraic stack, there is a commutative ring structure on $\cH_*(\cX)$ such that for any smooth morphism $f:\cY\to \cX$, the smooth pullback $f^* : \cH_*(\cX) \to \cH_{*+e}(\cY)$ is a ring homomorphism.
Indeed,
for each smooth $t:T\to \cX$ with $T$ a quasi-projective scheme, we have the intersection ring $H_*(T)$ because $T$ is a smooth scheme. Since the lci pullbacks for schemes preserve the intersection ring structure, we obtain the desired ring structure on $H_*(\cX)$.

\medskip
\noindent \underline{Formal group law}: 
If $F_H(u,v) \in  H_*(\spec \bk)[[u,v]]$ is the formal group law in \eqref{2.74} for $\mathbf{QSch}_\bk$, the same formula 
\beq \label{2.21}
c_1(L_1\otimes L_2) =F_H(c_1(L_1),c_1(L_2))
\eeq
holds for any line bundles $L_1$ and $L_2$ over an algebraic stack $\cX$. For stacks, Chern classes may not be nilpotent but the equation \eqref{2.21} still makes sense since $c_1(t^*L_1) : H_*(T) \to H_{*-1}(T)$ is nilpotent for each $t : T \to \cX$.

\medskip
\noindent \underline{Chern classes}: 
For a vector bundle $E$ over an algebraic stack $\cX$, the $i$-th Chern class
$$c_i(E) : \cH_*(\cX) \lra \cH_{*-i}(\cX)$$
can be defined by the limit as in \eqref{2.9}. Also, if $p : \PP E \to \cX$ is the associated projective bundle, then we have {$\sum_{i=0}^r (-1)^i c_i(p^*E\dual)\cdot c_1(\sO_{\PP E}(1))^{r-i}=0.$}

\medskip
\noindent \underline{Gysin map for a vector bundle stack}: 
For a vector bundle stack $\fE$ over an algebraic stack $\cX$, we define the \emph{Gysin map}
\beq\label{2.14} 0^!_{\cE}=0_{\cE}^* : \cH_*(\cE) \lra \cH_{*+r}(\cX)\eeq
as follows. 
First, assume that $\fE$ is globally presented, i.e. $\fE=[E_1/E_0]$ for some homomorphism $E_0 \to E_1$ of locally free sheaves. Then
\eqref{2.14} is defined by the smooth pullbacks 
$$\cH_{*}(\cE)\lra \cH_{*+r_0}(E_1) \longleftarrow \cH_{*+r_0-r_1} (\cX) $$ 
which are isomorphisms by the extended homotopy property since $E_1\to\cE$ is a $E_0$-torsor and $E_1 \to \cX$ is a $E_1$-torsor. Here $r_0$ and $r_1$ are the ranks of $E_0$ and $E_1$ respectively, and $r=r_0-r_1$ is the rank of the complex $[E_0\to E_1]$. 
In general, for any smooth morphism $t : T \to \cX$ from a quasi-projective scheme $T$, the vector bundle stack $t^*\fE$ is globally presented by the resolution property of the quasi-projective scheme $T$. The Gysin maps $0_{t^*\fE}^* : \cH_*(t^*\fE) \to H_{*+r}(T)$ for all $t:T\to \cX$ give us the desired map $0^*_\cE$.

\medskip

\subsection{Good system of approximations}\label{S4.2}

In order to make the limit theory $\cH_*$ a (weak) intersection theory for stacks, we further need the refined Gysin pullbacks and exterior products. To make sense of these maps, we confine ourselves to a smaller category of algebraic stacks.

\begin{defi}[Good system of approximations] \label{3.1}
Let $H_*$ be an intersection theory for $\mathbf{QSch}_\bk$. A \emph{good system of approximations for an algebraic stack $\cX$ with respect to $H_*$} consists of morphisms
\beq\label{3.2} \{x_i:X_i\to \cX\}_{i\ge 0},\quad \{x_{i,i+1}:X_i\to X_{i+1}\}_{i\ge 0}\eeq
such that 
\begin{enumerate}
\item $x_{i+1}\circ x_{i,i+1}$ and $x_i$ are 2-isomorphic;
\item $x_i$ is smooth of relative dimension $d(x_i)$ and $X_i\in \mathbf{QSch}_\bk$;
\item for any quasi-projective scheme $S$ and a quasi-projective morphism $S\to \cX$, the natural map
$$H_d(S) \lra \varprojlim_{i}H_{d+d(x_i)}(S\times_\cX X_i)$$
is an isomorphism for all d;
\item for any quasi-projective morphism $\cY\to \cX$ of algebraic stacks, the lci pullback
$$H_{*+d(x_{i+1})}(\cY\times_\cX X_{i+1})\lra H_{*+d(x_i)}(\cY\times_\cX X_i)$$
induced by $x_{i,i+1}$ is surjective. 
\end{enumerate}
If (1)-(4) hold for any intersection theory $H_*$ for $\mathbf{QSch}_\bk$, we say \eqref{3.2} is a \emph{good system of approximations}.

We let $\mathbf{St}_\bk^{ga}\subset \mathbf{St}_\bk$ be the full 2-subcategory of $\cX\in  \mathbf{St}_\bk$ such that
\begin{enumerate}
\item the diagonal $\Delta_{\cX} : \cX \to \cX \times \cX$ is quasi-projective and 
\item $\cX$ has a good system of approximations.  
\end{enumerate}
\end{defi}

For our discussion below, the following simple lemma will be useful.
\begin{lemm}\label{2.78}
For a quasi-projective morphism $f : \cX \to \cY$ of algebraic stacks in $\mathbf{St}_\bk$, if $\{Y_i \to \cY\}$ is a good system of approximations, so is $\{X_i=\cX \times_{\cY}Y_i \to \cX\}$.

For algebraic stacks $\cX$ and $\cY$ in $\mathbf{St}_\bk^{ga}$, if $\{X_i\to\cX\}$ and $\{Y_i\to \cY\}$ are good systems of approximations, then so is $\{X_i\times Y_i \to \cX \times\cY\}$. 
\end{lemm}
The proof of this lemma is rather elementary and we omit it. 

\begin{prop}\label{3.3}
The 2-category $\mathbf{St}_\bk^{ga}$ is an admissible subcategory of $\mathbf{St}_\bk$.
\end{prop}

\begin{proof}
By Lemma \ref{2.78}, $\cX$ admits a good system of approximations if {$\cY \in \mathbf{St}_\bk^{ga}$} and the morphism $\cX\to\cY$ is quasi-projective. 
For morphisms $\cX\to \cZ$, $\cY\to \cZ$ of algebraic stacks in $\mathbf{St}_\bk^{ga}$, 
from the fiber diagram
$$\xymatrix{
\cX\times_\cZ\cY\ar[r] \ar[d] & \cX\times \cY\ar[d]\\
\cZ\ar[r]^{\Delta_\cZ} & \cZ\times \cZ,
}$$
we find that the top horizontal arrow is quasi-projective and hence $\cX\times_\cZ\cY$ also admits a good system of approximations by Lemma \ref{2.78}.

If $\cX\to \cY$ is quasi-projective and the diagonal $\Delta_{\cY}$ of $\cY$ is quasi-projective, then so is the diagonal $\Delta_{\cX}$ of $\cX$. Indeed, $\cX\times_\cY\cX\to \cX\times \cX$ is quasi-projective since $\Delta_{\cY}$ is quasi-projective and the diagonal $\Delta_{\cX/\cY} : \cX\to \cX\times_\cY \cX$ is a closed immersion. The quasi-projectivity of the diagonal of $\cX\times_\cZ\cY$ is similar. 
\end{proof}

\begin{exam}[Totaro's approximation of classifying spaces] \label{3.4} 
Let $G$ be a linear algebraic group over $\bk$. By \cite[Remark 1.4]{Tot}, there is a sequence of $G$-representations $V_i$ and $G$-invariant open $U_i\subset V_i$ such that
\begin{enumerate}
\item stack quotient $[U_i/G]$ is a quasi-projective scheme;
\item $V_{i+1}=V_i\oplus W_i$ for some $G$-representation $W_i$;
\item $U_i\times W_i\subset U_{i+1}$;
\item $\mathrm{codim}_{V_i}(V_i- U_i)<\mathrm{codim}_{V_{i+1}}(V_{i+1}- U_{i+1})$.
\end{enumerate}
The quotients $\{U_i/G\to BG\}$ is a good system of approximations for the classifying stack $BG=[\mathrm{pt}/G]$  by the proof of \cite[Proposition 15]{HeMa}. Moreover the diagonal of $BG$ is affine.

Let $X$ be a quasi-projective scheme with a linear action of a linear algebraic group $G$. Then the quotient stack $\cX=[X/G]$ is quasi-projective over $BG$ and $\cX$ admits a good system of approximations 
$$X\times_G U_i\lra \cX$$
with $U_i$ above. It is easy to see that the diagonal of $\cX=[X/G]$ is quasi-projective. 
Therefore the quotient stack $\cX$ of a quasi-projective scheme by a linear algebraic group lies in $\mathbf{St}_\bk^{ga}$.
\end{exam}

When there is a good system $\{x_i:X_i\to \cX\}$ of approximations, the limit theory $\cH(\cX)$ is the inverse limit of $H_*(X_i)$ only. 
\begin{prop}\label{3.6}
Let $\cX\in \mathbf{St}_\bk^{ga}$ be an algebraic stack equipped with a good system $\{x_i:X_i\to \cX\}$ of approximations. Then the smooth pullbacks $x_i^*:\cH_d(\cX)\to \cH_{d+d(x_i)}(X_i)=H_{d+d(x_i)}(X_i)$ induce an isomorphism
\beq \label{3.7} 
\cH_d(\cX)\mapright{\cong} \varprojlim_i H_{d+d(x_i)}(X_i)
\eeq
for all $d$. In particular, the right hand side of \eqref{3.7} is independent of the choice of good approximations. 
\end{prop}

\begin{proof} 
Suppose $\xi=\{\xi_i\}_{i\ge 0}\in \varprojlim_{i} H_{d+d(x_i)}(X_i)$. For any smooth morphism $t:T\to \cX$ with $T$ quasi-projective, we have an isomorphism
$$H_{d+d(t)}(T)\cong \varprojlim_{i} H_{d+d(x_i)}(T\times_\cX X_i).$$
The smooth pullbacks by $T\times_\cX X_i\to X_i$ induce elements $\eta_t\in H_{*+d(t)}(T)$ by the isomorphism.
It is straightforward to check that 
$$\eta=\{\eta_t\}\in \varprojlim_{t}H_{d+d(t)}(T)=\cH_d(\cX)$$
and the map $\xi\mapsto \eta$ is the inverse of \eqref{3.7}. 
\end{proof}

We obtain the following from Example \ref{3.4} and Proposition \ref{3.6}.

\begin{coro}[Equivariant intersection theories]\label{3.8}
Let $X$ be a quasi-projective scheme equipped with a linear action of a linear algebraic group $G$. Let $\cX=[X/G]$ be the quotient stack. Then 
$$\cH_d(\cX)=\varprojlim_i H_{d+d(x_i)}(X\times_G U_i)$$
for all $d$ where $U_i$ and $x_i$ are from Example \ref{3.4}. In particular, the right hand side is independent of the presentation $\cX=[X/G]$ as a global quotient.
\end{coro}

\begin{exam}[Equivariant Chow and algebraic cobordism theories]\label{2.77} 
When the algebraic stack $\cX$ admits a presentation as the global quotient $[X/G]$ of a quasi-projective scheme $X$ by a linear algebraic group $G$, 
Corollary \ref{3.8} tells us that the limit theory $\cH_*(\cX)$ coincides with the equivariant Chow theory for $\cX=[X/G]$ of Edidin-Graham in \cite{EdGr} (resp. the equivariant algebraic cobordism of Krishna and Heller-Malag\'{o}n-L\'{o}pez in \cite{HeMa, Kri}) when $H_*$ is {Chow theory} in \cite{Ful} (resp. algebraic cobordism of Levine-Morel in \cite{LeMo}). \end{exam}
Thus the limit theory $\cH_*$ of quotient stacks generalizes all known equivariant intersection theories. 
It also proves that the equivariant theories are independent of the presentation of $\cX$ as a quotient stack $[X/G]$ for any $\cH_*$. 

\medskip

For virtual intersection theories below, we will use cone stacks.  
\begin{exam}[Cone stack]\label{3.15}
Let $\cX\in\mathbf{St}_\bk^{ga}$ and $\fC=[C/E]$ be a globally presented cone stack for some vector bundle $E$ and a $E$-cone $C$ (cf.  \cite{BeFa}). 
As $C$ is quasi-projective over $\cX$, $C $ admits a good system $\{C_i \to C\}$ of approximations by Lemma \ref{2.78}. 
Since $C$ is an $E$-torsor over $\fC$, the good system $\{C_i \to C\}$ for $C$ induces a good system $\{C_i \to C \to \fC\}$ of approximations for $\fC$ by the extended homotopy (EH).  Therefore $\fC \in \mathbf{St}_\bk^{ga}$. 
\end{exam}

\subsection{An example}\label{3.9}
Is there an algebraic stack admitting a good system of approximations which is not a global quotient stack?

Let $\cX$ be a smooth algebraic stack with open substacks 
$\{\cX_i\}_{i\ge 0}$ satisfying 
\begin{enumerate}
\item $\cX_i\subset \cX_{i+1}$ for $i\ge 0$ and $\cX=\cup_i \cX_i$;
\item each $\cX_i$ admits a good system of approximations $\{X_{i,j}\}_{j\ge 0}$;
\item there are morphisms $X_{i,j}\to X_{i+1,j}$ which factor as 
$$X_{i,j}\mapright{0_{E_{i,j}}} E_{i,j}\mapright{\text{open}} X_{i+1,j}$$ 
for some vector bundles $E_{i,j}$ over $X_{i,j}$ and
fit into commutative
\beq\label{e4}\xymatrix{
X_{i,j}\ar[r]\ar[d] &X_{i+1,j}\ar[d]\\
X_{i,j+1}\ar[r] & X_{i+1,j+1}
}\eeq
where the vertical arrows are the structural morphisms of good approximations.
\end{enumerate}
Then it is immediate to see that $\cX$ admits a good system of approximations 
$$\{X_{i,i} \lra \cX_i\subset \cX\}_{i\ge 0}.$$

For example, let $\cM=\cM_C(r,d)$ denote the moduli stack of vector bundles of rank $r$ and degree $d$ over a smooth projective curve $C$ of genus $g$ over $\CC$. For each $m$, fix an isomorphism $\CC^{m+1}\cong \CC^{m}\times \CC$.

Fix a very ample line bundle $\sO_C(1)$ on $C$ and let $a$ be the degree of $\sO_C(1)$.
Let $\ell=\chi(\sO_C(1))=a-g+1$ so that we have a surjection $\sO_C^{\oplus \ell}\to \sO_C(1)$. 
Let $P_0=d-r(g-1)=\chi(E)$ be the Riemann-Roch number of $E\in \cM$  and let $P_n=P_0+ran=\chi(E(n))$. 
Consider the quot scheme $Quot^P(\sO_C^{\oplus P_n}(-n))$ of surjective homomorphisms
$\sO_C^{\oplus P_n}(-n)\to E$ with $E$ coherent of $\chi(E(k))=P_k$ for all $k$. 
Let $Y_n$ be the open subscheme of $Quot^P(\sO_C^{\oplus P_n}(-n))$ parameterizing quotients
$q:\sO_C^{\oplus P_n}(-n)\to E$ such that 
$H^1(C,E(n-1))=0$ and $H^0(q(n)):H^0(\sO_C^{\oplus P_n})\to H^0(E(n))$ is an isomorphism.
In particular, $E$ is $n$-regular so that the natural map
\beq\label{e1} H^0(C,E(n))\otimes H^0(C,\sO_C(1))\lra H^0(C, E(n+1))\eeq
is surjective. Let $G_n=GL(P_n)$. Then it is easy to see that
$\cM_n:=[Y_n/G_n]$ is open in $\cM$ and $\cup_n \cM_n=\cM$ (cf. \cite{New}).

For $m\ge P_n$, let $\mathrm{Hom}(\CC^m, \CC^{P_n})^\circ$ denote the set of surjective homomorphisms from $\CC^m$ to $\CC^{P_n}$. 
By Example \ref{3.4},  the morphisms 
$$Y_{n,m}=Y_n\times_{G_n} \mathrm{Hom}(\CC^m, \CC^{P_n})^\circ\lra [Y_n/G_n]=\cM_n$$
form a good system of approximations for $\cM_n$ and $Y_{n,m}$ parameterizes surjective homomorphisms
$\sO_C^{\oplus m}(-n)\to E$ that factor as 
\beq\label{e2} \sO_C^{\oplus m}(-n)\twoheadrightarrow \sO_C(-n)\otimes H^0(E(n))=\sO_C^{\oplus P_n}(-n)\twoheadrightarrow E.\eeq
By \eqref{e1}, we have a morphism $Y_{n,m}\to Y_{n+1,m\ell}$ by sending \eqref{e2} to
\beq\label{e3}\sO_C^{\oplus m\ell}(-n-1)=\sO_C^{\oplus m}(-n)\otimes \sO_C^{\oplus \ell}(-1)\twoheadrightarrow\eeq 
$$\sO_C^{\oplus P_n}(-n)\otimes \sO_C^{\oplus \ell}(-1) \twoheadrightarrow 
\sO_C^{\oplus P_{n+1}}(-n-1)\twoheadrightarrow E.$$

Fix a number $\nu\ge P_0$. 
Let $X_{i,j}=Y_{i,\nu\ell^{i+j}}$ so that $\{X_{i,j}\to \cM_i\}_j$ is a good system of approximations and  
\eqref{e3} gives us morphisms $X_{i,j}\to X_{i+1,j}$ that fit into the commutative diagram \eqref{e4}. 
There is a vector bundle $E_{i,j}\to X_{i,j}$ whose fiber over \eqref{e2} is $\Hom(\CC^{\nu\ell^{i+j+1}}, \CC^{\ell P_i-P_{i+1}})$ such that $E_{i,j}$ is open in $X_{i+1,j}$.
We thus find that
$\cM\in \mathbf{St}_\bk^{ga}$.  

%
%
%

\bigskip

\subsection{Weak intersection theories for $\mathbf{St}_\bk^{ga}$}\label{S4.3}

The purpose of this subsection is to prove the following.

\begin{theo} \label{2.10}
Let $H_*$ be an intersection theory for $\mathbf{QSch}_\bk$.
Then the limit theory $\cH_*$ defined by \eqref{2.6} is a weak intersection theory for $\mathbf{St}_\bk^{ga}$ (cf. Definition \ref{2.70}). 
\end{theo}

We need to define refined Gysin pullbacks and exterior products.

\begin{defi}[Refined Gysin pullback for $\cH_*$]\label{3.10}
Let  $f : \cX \hookrightarrow \cY$ be a regular immersion of codimension $c$ that fit into a Cartesian square 
\beq\label{e7}\xymatrix{
\cX'\ar[r]^{f'}\ar[d]_{g'} & \cY'\ar[d]^g\\
\cX\ar[r]^f &\cY
}\eeq
in $\mathbf{St}_\bk^{ga}$.
First let us assume that $g$ is quasi-projective. For a good system $\{y_i:Y_i \to \cY\}_i$ of approximations, the induced
Cartesian square
$$\xymatrix{
\cX'\times_\cY Y_i\ar[r]^{{f'_i}}\ar[d]_{g'} & \cY'\times_\cY Y_i\ar[d]^g\\
\cX\times_\cY Y_i\ar[r]^{{f_i}} & Y_i
}$$
gives us the refined Gysin pullback 
$$f^!:\cH_d(\cY')=\varprojlim_i H_{d+d(y_i)}(\cY'\times_\cY Y_i)\lra \varprojlim_i H_{d+d(y_i)-c}(\cX'\times_\cY Y_i)=\cH_{d-c}(\cX')$$ 
by Lemma \ref{2.78}. 
In general, for a good system $\{y'_i : Y_i' \to \cY'\}$ of approximations, the compositions $g\circ y'_i : Y'_i \to \cY$ are quasi-projective. Hence, the refined Gysin pullbacks $f^!:H_*(Y_i')\to H_{*-c}(\cX'\times_{\cY'}Y'_i)$ for all $i$ give us the \emph{refined Gysin pullback}
$$ f^! : \cH_d(\cY')=\varprojlim_i H_{d+d(y'_i)}(Y'_i) \lra \varprojlim_{i} H_{d+d(y'_i)-c}(\cX'\times_{\cY'}Y'_i)=\cH_{d-c}(\cX').$$
\end{defi}
By the arguments in the proof of Proposition \ref{3.6}, it is straightforward that 
the refined Gysin pullback does not depend on the choice of a good system of approximations. Also, these refined Gysin pullbacks are functorial. 

\begin{defi}[Exterior product for $\cH_*$]\label{3.11}
Let $\cX$ and $\cY$ be algebraic stacks in $\mathbf{St}_\bk^{ga}$. For good systems $\{x_i:X_i\to \cX\}$ and $\{y_i:Y_i\to \cY\}$ of approximations, 
and for  $\xi\in H_d(\cX)$ and $\eta\in H_e(\cY)$, we let 
$\{\xi_i\in H_{d+d(x_i)}(X_i)\}$ and $\{\eta_i\in H_{e+d(y_i)}(Y_i)\}$ be the images of $\xi$ and $\eta$ respectively, by \eqref{3.7}. Using Proposition \ref{3.6} and Lemma \ref{2.78}, the \emph{exterior product} is defined by
$$\times:\cH_*(\cX)\otimes \cH_*(\cY)\lra \cH_*(\cX\times \cY), \quad (\xi,\eta)\mapsto \{\xi_i\times \eta_i\}$$
where we used the exterior product $\times:H_*(X_i)\otimes H_*(Y_i)\to H_*(X_i\times Y_i)$ for quasi-projective schemes.  
\end{defi}

We leave it as an exercise to check that the exterior product in Definition \ref{3.11} is independent of the choice of good systems of approximations.

All the axioms for a weak intersection theory in Definition \ref{2.70} are easy to check for $\cH_*$ except possibly for the weak excision property.
\begin{prop} \label{3.12} The weak excision axiom (WEx) in Definition \ref{2.70} holds for $\cH_*(\cX)$ with $\cX\in \mathbf{St}_\bk^{ga}$. 
\end{prop}

\begin{proof}
Let $\imath : \cZ \hookrightarrow\cX$ be a closed immersion of algebraic stack in $\mathbf{St}_\bk^{ga}$ and $\jmath:\cU=\cX-\cZ\hookrightarrow \cX$ denote its complement. 
Fix a good system $$\{x_i:X_i \lra \cX\}, \quad \{\varphi_i:X_i\lra X_{i+1}\}$$ of approximations so that
$Z_i=\cZ\times_{\cX} X_i\to \cZ$ and $U_i=\cU\times_{\cX}X_i\to \cU$ are good systems of approximations by Lemma \ref{2.78}. By the excision property of $H_*$ for quasi-projective schemes, we have an exact sequence
$$H_{d+d(x_i)} (Z_i) \mapright{{\imath_i}_*} H_{d+d(x_i)} (X_i) \mapright{\jmath_i^*} H_{d+d(x_i)} (U_i) \lra 0$$
where $\imath_i: Z_i \hookrightarrow X_i$  and $\jmath_i: U_i \hookrightarrow X_i$ denote the induced inclusions. 
Let $K_i=\mathrm{Im}({\imath_i}_*)$. 
By Definition \ref{3.1} (4), $(\varphi_i|_{Z_i})^*:H_{d+d(x_{i+1})} (Z_{i+1})\to H_{d+d(x_i)} (Z_i) $ is surjective for all $i$ and hence $\{K_{i+1} \to K_i\}$ is a surjective inverse system. 
Hence the exact sequences $0\to K_i\to H_{d+d(x_i)} (X_i) \to H_{d+d(x_i)} (U_i) \to 0$ induce
an exact sequence
$$0 \lra \varprojlim_i K_i \lra \cH_d (\cX)\lra \cH_d (\cU) \lra 0.$$
Therefore, $\jmath^* : \cH_*(\cX) \lra \cH_*(\cU)$ is surjective.

When $\imath$ is regular of codimension $r$, the homomorphism $\lambda_{\cZ/\cX}:\cH_*(\cU)\to \cH_{*-r}(\cZ)$ follows from taking the limit of $\lambda_{Z_i/X_i}$. 
\end{proof}
This completes a proof of Theorem \ref{2.10}. 

%

\bigskip

\section{Virtual pullbacks}\label{S8}

In this section, we generalize the virtual pullbacks in \cite{Man} to
\begin{enumerate}
\item weak intersection theories $H_*$ for an admissible category $\bcV$ of algebraic stacks with deformation spaces for \DM type morphisms and Gysin maps for vector bundle \emph{stacks} and
\item the limit intersection theories $\cH_*$ for $\mathbf{St}_\bk^{ga}$ arising from intersection theories for quasi-projective schemes.
\end{enumerate}  
The first case generalizes the constructions of Manolache in \cite{Man} for {Chow theory} and Qu in \cite{Qu} for {$K$-theory}.
The second case generalizes the construction of virtual fundamental class by Shen in \cite{Shen} for algebraic cobordism of quasi-projective schemes to quasi-projective \DM stacks (in the sense of \cite{KreG}).

\medskip

\subsection{Weak intersection theories for algebraic stacks}\label{S8.1}
Throughout this subsection, we assume the following.
\begin{assu}\label{2.84}
Let $H_*$ be a weak intersection theory, possibly without the axioms (EH) and (PB), for 
an {admissible} category 
$\bcV$ in $\mathbf{St}_\bk$ satisfying the following:
\begin{enumerate}
\item (Deformation spaces) for a \DM type morphism $f:X\to Y$ in $\bcV$, the deformation space $M^\circ_{X/Y}$ lies in $\bcV$;
\item (Gysin maps for vector bundle stacks) for a vector bundle stack $\cE$ over $X \in \bcV$, $\cE$ lies in $\bcV$ and there is a \emph{Gysin map}
$$0^!_{\cE} : H_*(\cE) \lra H_{*+r}(X)$$
satisfying
\begin{enumerate}
\item $0^!_{\cE}\circ \pi^*=\id_{H_*(X)}$ where $\pi:\cE\to X$ is the projection;
\item $f_* \circ 0^!_{f^*\cE} = 0^!_{\cE}\circ f'_*$ for projective $f:Y\to X$ where $f':f^*\cE\to\cE$ is the induced map;
\item $f^* \circ 0^!_{\cE} = 0^!_{f^* \cE}\circ f'^*$ for smooth $f:Y\to X$;
\item $f^! \circ 0^!_{\cE} = 0^!_{f^* \cE}\circ f^!$ for a regular immersion $f:Y\to X$;
\item $0^!_{\cE\times_X \cE'} = 0^!_{\cE}\circ 0^!_{\pi^*\cE'}$ for a vector bundle stack $\cE'$ over $X$.
\end{enumerate}
\end{enumerate}
\end{assu}


Note that both Kresch's Chow theory for algebraic stacks which admit stratifications by quotient stacks and the algebraic K-theory {of coherent sheaves} for all algebraic stacks satisfy Assumption \ref{2.84} (cf. \cite{Kre, Qu}). Also any limit intersection theory $\cH_*$ induced from an intersection theory $H_*$ for quasi-projective schemes has Gysin maps for vector bundle stacks \eqref{2.14}.

\medskip

Recall that the {deformation} space for a morphism $f:X\to Y$ is a flat morphism $M^\circ_{X/Y}\to\PP^1$ whose fiber over $0$ is the relative intrinsic normal cone $\fC_{X/Y}$ of $f$ and  whose restriction to $\PP^1-\{0\}$ is $M^\circ_{X/Y}-\fC_{X/Y}=Y\times\bbA^1$ (cf \cite[\S5.1]{Kre}). 

Also recall from \cite{BeFa, Man} that a \emph{perfect obstruction theory} $\phi:E\to \bbL_f$ for a \DM type morphism $f:X\to Y$ is a morphism in the derived category of quasi-coherent sheaves on $X$ such that 
\begin{enumerate}
\item $E$ is locally a two-term complex $[E^{-1}\to E^0]$ of locally free sheaves;
\item $\bbL_f$ is the cotangent complex of $f$ truncated to the interval $[-1,0]$;
\item $h^0(\phi)$ is an isomorphism and $h^{-1}(\phi)$ is surjective.
\end{enumerate}
By \cite{BeFa, Man}, a perfect obstruction theory gives us a closed embedding
\beq\label{3.31}\imath_{X/Y}:\fC_{X/Y}\hookrightarrow \cE=h^1/h^0(E^\vee)\eeq
of the intrinsic normal cone 
into the vector bundle stack $\cE$ which is locally $[E_1/E_0]$ where $[E_0\to E_1]$ is the dual of a local presentation $[E^{-1}\to E^0]$ of $E$ by locally free sheaves. 

Under Assumption \ref{2.84}, we have the specialization homomorphism \eqref{2.80}
\beq\label{2.90}\mathrm{sp}_{X/Y}:H_*(Y)\lra H_*(\fC_{X/Y}).\eeq

\begin{defi} \label{2.86}
Under Assumption \ref{2.84}, the virtual pullback $f^!$ by a \DM type morphism $f:X\to Y$ equipped with a perfect obstruction theory $\phi:E\to \bbL_f$ is defined as  
\beq\label{2.82} f^!:H_*(Y)\mapright{\mathrm{sp}} H_*(\fC_{X/Y})\mapright{{\imath_{X/Y}}_*}
H_*(\cE)\mapright{0^!_\cE} H_{*+r}(X)\eeq
where $r$ is the rank of $E$. 
For a \DM stack $X$ with an embedding $\fC_X\hookrightarrow \cE$ of the normal cone of $X$ into a vector bundle stack $\cE$ on $X$, the morphism $p:X\to \Spec \bk$ defines the \emph{virtual fundamental class}
\beq\label{2.81}[X]\virt=p^!\mathbf{1}\in H_*(X)\eeq 
where $\mathbf{1}\in H_0(\Spec \bk)$ is the unit in Definition \ref{2.1} (i). 
\end{defi}
\begin{rema}
When $H_*$ is Kresch's Chow theory, $[X]\virt$ is the virtual fundamental class in \cite{BeFa, LiTi, Kre} and the virtual pullback $f^!$ was defined in \cite{Man}. When $H_*$ is the K-theory of coherent sheaves, $[X]\virt$ is the virtual structure sheaf in \cite{BeFa, YLe} and the virtual pullback $f^!$ was defined in \cite{Qu}.  
\end{rema}

\begin{theo}\label{2.85} Under Assumption \ref{2.84}, we have the following. 
\begin{enumerate}
\item The virtual pullback \eqref{2.82} commutes with projective pushforwards, smooth pullbacks and  other virtual pullbacks. 
\item The virtual pullback is functorial: if $f:X\to Y$ and $g:Y\to Z$ are \DM type morphisms in $\bcV$ with perfect obstruction theories $\phi_f:E_f\to \bbL_f$ and $\phi_g:E_g\to \bbL_g$ that  fit into a commutative diagram of exact triangles
\beq\label{2.87}\xymatrix{
E_g|_X\ar[r]\ar[d] & E_{g\circ f}\ar[r]\ar[d] & E_f\ar[r]\ar[d] &\\
\bbL_g|_X \ar[r] & \bbL_{g\circ f}\ar[r] & \bbL_f \ar[r] &
}\eeq
where the middle vertical arrow is a perfect obstruction theory of $g\circ f$, then 
$$f^!\circ g^!=(g\circ f)^!.$$
\item If the perfect obstruction theories for $f:X\to Y$ and $Y\to\Spec \bk$ satisfy the condition in (2), we have the equality
\beq\label{2.83}f^![Y]\virt=[X]\virt.\eeq
\end{enumerate}
\end{theo}
\begin{proof}
The exactly same arguments in \cite{Qu} which in turn are adapted from the arguments in \cite{Ful, Man, CKL} hold for weak intersection theories under Assumption \ref{2.84}. We omit the details. 
\end{proof} 
Note that we don't need (EH) and (PB) for virtual pullbacks and virtual fundamental classes in this subsection.

\medskip

\subsection{Virtual pullback for quasi-projective schemes} \label{S8.4}

Unfortunately, Assumption \ref{2.84} is not automatically satisfied because usually an intersection theory of schemes
does not extend to stacks except for a few cases.  
Moreover, for a \DM type morphism $f:X\to Y$ in 
$\bcV$, the deformation space $M^\circ_{X/Y}$ should lie in $\bcV$. 
Hence the standard arguments in \S\ref{S8.1} do not give us the virtual pullbacks and the virtual fundamental classes for quasi-projective schemes or limit intersection theories in general. 

In this subsection, we will directly construct the virtual pullbacks and virtual fundamental classes 
for intersection theories on $\mathbf{QSch}_\bk$.
In the subsequent subsection, we will extend the construction 
to limit intersection theories on $\mathbf{St}^{ga}_\bk$. 
 
As in \S\ref{S4}, we will denote schemes by roman characters $X, Y, Z$ and stacks by calligraphic $\cX, \cY, \cZ$ to distinguish them. 

\medskip
 
Throughout the rest of this section, 
we \emph{let $H_*$ be an intersection theory for the category $\mathbf{QSch}_\bk$ of quasi-projective schemes and let $\cH_*$ denote the limit intersection theory for $\mathbf{St}_\bk^{ga}$} defined by \eqref{2.6}.

\medskip

We begin our construction of the virtual pullback 
with the specialization homomorphisms for quasi-projective schemes.

Let $f : X \to Y$ be a morphism of quasi-projective schemes. Then $f$ can be factored as 
the composition
$$f : X \lra Z \mapright{g} Y$$
of a closed immersion into a quasi-projective scheme $Z$ and a smooth morphism $g$ of relative dimension $e$. 
Since the deformation space $M^\circ_{X/Z}$ is a quasi-projective scheme \cite[Chapter 5]{Ful}, 
we have the specialization homomorphism (cf. \eqref{2.18})
$$\mathrm{sp}_{X/Z}: H_*(Z)\lra H_*(C_{X/Z}).$$

The intrinsic normal cone of $f$ is by definition the quotient 
$$\fC_{X/Y}=[C_{X/Z}/\bbT_{Z/Y}|_X]$$ of the normal cone $C_{X/Z}$ by the relative tangent bundle $\bbT_{Z/Y}$ of $g$.
Hence the quotient map
$$\pi : C_{X/Z} \lra \fC_{X/Y}$$
is a $\bbT_{Z/Y}|_X$-torsor. 
By the extended homotopy for $\cH_*$, the smooth pullback 
$$\pi^*:\cH_*(\fC_{X/Y})\lra \cH_{*+e}(C_{X/Z})=H_{*+e}(C_{X/Z})$$ 
is an isomorphism.
\begin{defi} \label{4.1}
For a morphism $f:X\to Y$ of quasi-projective schemes, the \emph{specialization homomorphism}  is  the composition
\beq \label{4.21}
\mathrm{sp}_{X/Y} : H_*(Y) \mapright{g^*} H_{*+e}(Z) \mapright{\mathrm{sp}_{X/Z}} H_{*+e}(C_{X/Z}) \mapright{(\pi^*)^{-1}} \cH_*(\fC_{X/Y}).
\eeq
\end{defi}

The specialization homomorphism  $\mathrm{sp}_{X/Y}$ is well defined.
\begin{lemm}\label{4.2}
$\mathrm{sp}_{X/Y}$ is independent of the factorization $X \to Z \to Y$.
\end{lemm}

\begin{proof}
Choose another factorization $X \mapright{} Z' \mapright{g'} Y$ of $f$ by a closed immersion and a smooth morphism for some quasi-projective scheme $Z'$. After replacing $Z'$ by $Z\times_Y Z'$, we may assume that there is a smooth morphism $a : Z' \to Z$ making the diagram
$$\xymatrix{
& Z'\ar[d]^a \ar[rd]^{g'} & \\
X \ar[ru]^{} \ar[r]^{} & Z \ar[r]^g & Y
}$$
commute. As $M^\circ_{X/Z'}$ is open in the blowup of $Z'\times \PP^1$ along $X\times \{0\}$, 
we have the morphism 
\beq\label{4.50}M^\circ_{X/Z'}\lra Z'\times \PP^1\mapright{a\times\id} Z\times \PP^1\eeq
and the inverse image of $X\times \{0\}$ is the Cartier divisor $C_{X/Z'}$. Hence
\eqref{4.50} factors through a morphism 
$ M^\circ_{X/Z'}\lra M^\circ_{X/Z} $
and we have the commutative diagram of Cartesian squares
$$\xymatrix{
Z'\ar[d]_a \ar[r] & M^{\circ}_{X/Z'} \ar[d] & C_{X/Z'} \ar[d]^b\ar[l]\\
Z\ar[r] & M^{\circ}_{X/Z}  & C_{X/Z}.\ar[l]
}$$
Since  Gysin pullbacks commute with smooth pullbacks, we have the equality
$$ b^* \circ \mathrm{sp}_{X/Z} = \mathrm{sp}_{X/Z'} \circ a^*. $$

The morphism $b$ is smooth by the exact sequence
\[ 0\lra \bbT_{Z'/Z}|_X\lra C_{X/Z'} \mapright{b} C_{X/Z}\lra 0\]
and the quotient $\pi':C_{X/Z'}\to \fC_{X/Y}$ by $\bbT_{Z'/Y}|_X$ equals $\pi\circ b$.
Therefore, $b^*\circ \pi^*=(\pi')^*$ and so
$$(\pi^*)^{-1}\circ \mathrm{sp}_{X/Z}\circ g^*=({\pi'}^*)^{-1}\circ b^*\circ \mathrm{sp}_{X/Z}\circ g^*$$
$$=({\pi'}^*)^{-1}\circ\mathrm{sp}_{X/Z'} \circ a^*\circ g^*=({\pi'}^*)^{-1}\circ\mathrm{sp}_{X/Z'} \circ  {g'}^*$$
as desired.
\end{proof}

We define the virtual pullback for a morphism of quasi-projective schemes as follows.
\begin{defi}\label{3.32}
Let $f:X\to Y$ be a morphism of quasi-projective schemes equpped with 
an embedding of the intrinsic normal cone $\fC_{X/Y}\hookrightarrow \cE$ into the vector bundle stack $\cE=h^1/h^0(E^\vee)$. Then the \emph{virtual pullback} is defined by
\beq\label{3.33}
f^!:H_*(Y)\mapright{\mathrm{sp}_{X/Y}} \cH_*(\fC_{X/Y})\mapright{{\imath_{X/Y}}_*} \cH_*(\cE)\mapright{0^!_\cE} \cH_{*+r}(X)=H_{*+r}(X)
\eeq
where the second arrow is the projective pushforward by the embedding \eqref{3.31}, $r$ is the rank of  $E$ and $0^!_\cE$ is \eqref{2.14}. For a quasi-projective scheme $X$ with an embedding $\fC_{X}\hookrightarrow \cE$ into a vector bundle stack, the morphism $p:X\to \Spec \bk$ defines the \emph{virtual fundamental class}
\beq\label{3.51} [X]\virt=p^!\mathbf{1}\in H_*(X)\eeq 
where $\mathbf{1}\in H_0(\Spec \bk)$ is the unit in Definition \ref{2.1} (i).
\end{defi}
\begin{rema}\label{4.3b}
(1) When $H_*$ is the algebraic cobordism in \cite{LeMo}, the virtual fundamental class $[X]\virt$ was defined by {Shen in \cite{Shen}}. In \cite{LoSc}, Lowrey-Sch{\"u}rg defined the virtual cobordism class when $X$ is the classical truncation of a quasi-smooth derived scheme.

(2) If the deformation space $M_{X/Y}^\circ\to \PP^1$ whose fiber over $t=0$ (resp. $t\ne 0$) is the intrinsic normal cone $\fC_{X/Y}=[C_{X/Z}/\bbT_{Z/Y}|_X]$ (resp. $Y$) lies in $\mathbf{St}^{ga}_\bk$, then  the specialization homomorphism \eqref{4.21} coincides with \eqref{2.90}. This follows from the construction of 
$M_{X/Y}^\circ$ as the groupoid $$[M_{X/R}^\circ\rightrightarrows M_{X/Z}^\circ]$$ where $R=Z\times_YZ$ \cite[Theorem 2.31]{Man}. 
Hence the virtual pullback \eqref{3.33} is a generalization of \eqref{2.82} to the case where $M_{X/Y}^\circ$
does not necessarily admit a good system of approximations. 

(3) Definition \ref{3.32} works for any weak intersection theory $H_*$ on an admissible category $\bcV$ of algebraic stacks and 
for a quasi-projective morphism $f:\cX\to \cY$ in $\bcV$ that factors as 
$$\cX\mapright{\jmath} \cZ\mapright{p} \cY$$
with $\jmath$ a closed immersion and $p$ a smooth quasi-projective morphism.
\end{rema}

The virtual pullback in Definition \ref{3.32} commutes with another virtual pullback and is functorial.

\begin{theo}\label{3.34} For any intersection theory $H_*$ for the category $\mathbf{QSch}_\bk$ of quasi-projective schemes, 
(1)-(3) in Theorem \ref{2.85} hold. 
\end{theo}
\begin{proof} (3) is immediate from (2). 
The commutativity of $f^!$ with projective pushforwards and smooth pullbacks follows from Lemma \ref{4.3} below.
The commutativity with other virtual pullbacks and functoriality (2) are proved by the usual arguments using the double deformation space in {\cite{Qu, Man, Kre, Kresch99, KKP, Ful}} together with Lemmas \ref{4.3} and \ref{4.3a}. 
\end{proof}

\begin{lemm}\label{4.3} 
Consider a Cartesian square
\beq\label{4.22}
\xymatrix{
X'\ar[r]^{f'}\ar[d]_h & Y' \ar[d]^g\\
X\ar[r]^f & Y
}
\eeq
of quasi-projective schemes. Then we have the following morphisms
$$k : \fC_{X'/Y'} \mapright{\jmath} h^*\fC_{X/Y}\to \fC_{X/Y},$$ 
$$\fC_{\fC_{X/Y}\times_Y Y'/\fC_{X/Y}} \mapright{a} \fC_{X/Y} \times_Y \fC_{Y'/Y} \mapleft{b} \fC_{\fC_{Y'/Y}\times_Y X/\fC_{Y'/Y}}$$
{where $\jmath$, $a$ and $b$ are closed immersions.}
\begin{enumerate}
\item If $g$ is projective, then
$ \mathrm{sp}_{X/Y} \circ g_* = k_* \circ \mathrm{sp}_{X'/Y'}.$
\item 
{If $g$ is smooth}, then $\jmath$ is an isomorphism and
$\mathrm{sp}_{X'/Y'} \circ g^* = k^*\circ \mathrm{sp}_{X/Y}.$
\item $a_* \circ \mathrm{sp}_{\fC_{X/Y}\times_Y Y'/\fC_{X/Y}} \circ \mathrm{sp}_{X/Y} = b_* \circ \mathrm{sp}_{\fC_{Y'/Y}\times_Y X/\fC_{Y'/Y}} \circ \mathrm{sp}_{Y'/Y}.$
\end{enumerate}
\end{lemm}

\begin{proof}
{For (1) and (2), it is easy to see that we may assume that $f$ is a closed immersion since smooth pullbacks are functorial and commute with projective pushforwards. Then the diagram
\beq\label{4.24}\xymatrix{
Y'\ar[d]_g \ar[r] & M^{\circ}_{X'/Y'} \ar[d] & C_{X'/Y'} \ar[d]^k\ar[l]\\
Y\ar[r] & M^{\circ}_{X/Y}  & C_{X/Y}\ar[l]
}\eeq
of two transversal Cartesian squares proves (1) and (2) as in \cite{Man,Qu}.}




The specialization homomorphism in (3) for the quasi-projective morphism $\fC_{X/Y}\times_Y Y' \to \fC_{X/Y}$ is defined by Remark \ref{4.3b} (3).
We may assume that $f$ and $g$ are closed immersions since specialization homomorphisms commute with smooth pullbacks. Then the usual argument in \cite{Ful, Man, Qu} using the double deformation space $M^{\circ}_{X/Y}\times_Y M^{\circ}_{Y'/Y}$ proves (3) since the deformation spaces are quasi-projective schemes.
\end{proof}

\begin{lemm}\label{4.3a}
Let $f:X\to Y$ and $g: Y \to Z$ be morphisms of quasi-projective schemes. Then two closed immersions 
$$\fC_{X/\fC_{Y/Z}} \mapright{a} \fC_{X\times\Po/M^{\circ}_{Y/Z}} \mapleft{b} \fC_{X/Z}$$ 
give us the identity
$$a_* \circ \mathrm{sp}_{X/\fC_{Y/Z}} \circ \mathrm{sp}_{Y/Z} = b_* \circ \mathrm{sp}_{X/Z}.$$
\end{lemm}

\begin{proof} Note that the specialization homomorphism for the quasi-projective morphism $ X\to \fC_{Y/Z}$ is defined by Remark \ref{4.3b} (3).
It is evident that we can form a commutative diagram
$$\xymatrix{
X \ar[r]_{f'} \ar[rd]_f & Y' \ar[d] \ar[r]_{g''}  & Z'' \ar[d]\\
 & Y \ar[rd]_g \ar[r]_{g'} & Z' \ar[d]\\
 & & Z
}$$
such that the horizontal arrows are closed immersions, the vertical arrows are smooth and the square is Cartesian. Then we have a {factorization}
$$X\times \Po \lra M^{\circ}_{Y'/Z''}\lra M^{\circ}_{Y/Z}$$
of $X\times \Po \to M^{\circ}_{Y/Z}$ by a closed immersion and a smooth morphism whose fiber over $0\in \Po$ is $X \to C_{Y'/Z''} \to \fC_{Y/Z}$. This induces a commutative diagram
$$\xymatrix{
C_{X/C_{Y'/Z''}} \ar[r]^{} \ar[d] & C_{X\times \Po/M^{\circ}_{Y'/Z''}} \ar[d] & C_{X/Z''} \ar[d] \ar[l]_{}\\
\fC_{X/\fC_{Y/Z}} \ar[r]^a & \fC_{X\times\Po/M^{\circ}_{Y/Z}} & \fC_{X/Z} \ar[l]_b
}$$
of cone stacks. Since the vertical arrows are torsors of vector bundles of the same rank, the two squares are Cartesian. It suffices to prove the lemma for the closed immersions $f':X\to Y'$ and $g'':Y'\to Z''$ since specialization homomorphisms commute with smooth pullbacks by Lemma \ref{4.3} (2). Then the usual arguments  in {\cite{Ful, KKP, Man, Qu}} using the double deformation space $M^{\circ}_{X\times \Po/M^{\circ}_{Y'/Z''}}$ remain valid since all the deformations spaces and the cone stacks are quasi-projective schemes in this case.
\end{proof}

\medskip

\subsection{Virtual pullback for limit intersection theory} \label{S8.5}
In this subsection, we extend the results in \ref{S8.4} to limit intersection theories for $\mathbf{St}^{ga}_\bk$. 

\medskip
We first define the virtual fundamental class of an algebraic stack. 
\begin{defi}\label{3.80}
Let $\cX$ be a \DM stack equipped with a perfect obstruction theory $\phi:E\to \bbL_\cX$. 
For a smooth morphism $t:T\to \cX$ from a quasi-projective scheme $T$, we say $\phi$ \emph{lifts to} 
a perfect obstruction theory $\phi_t:E_T\to \bbL_T$ if $\phi$ and $\phi_t$ 
fit into a commutative diagram of exact triangles
\beq\label{3.81}\xymatrix{
E|_{T}\ar[r] \ar[d]_{\phi} & E_{T}\ar[r]\ar[d]^{\phi_{t}} & \Omega_t\ar[r]\ar@{=}[d] & \\
\bbL_{\cX}|_{T}\ar[r] & \bbL_{T}\ar[r] & \Omega_t\ar[r] &   
}\eeq 
where $\Omega_t=\bbL_t$ is the cotangent bundle of $t$. 

Let $\{x_i:X_i\to \cX, \varphi_i:X_i\to X_{i+1}\}$ be a good system of approximations.
We say the perfect obstruction theory $\phi$ of $\cX$ is \emph{liftable} to the system $\{x_i\}$ if $\phi$ lifts to perfect obstruction theories $\phi_i:E_i\to \bbL_{X_i}$ that fit into 
a commutative diagram
\beq\label{3.85}\xymatrix{
E|_{X_i}\ar[dr] \ar[dd]\ar[rr]^{\phi} && \bbL_\cX|_{X_i}\ar[dr]\ar[dd]\\
& E_i\ar[rr]^(.3){\phi_i} && \bbL_{X_i}\\
E_{i+1}|_{X_i}\ar[ur]\ar[rr]^{\phi_{i+1}} && \bbL_{X_{i+1}}|_{X_i}\ar[ur]    
}\eeq 
\end{defi}

The octahedron axiom of the derived category tells us that $\phi_i$ and $\phi_{i+1}$ fit into 
 a commutative diagram of exact triangles
\beq\label{3.86}\xymatrix{
E_{i+1}|_{X_i}\ar[r] \ar[d]_{\phi_{i+1}} & E_{i}\ar[r]\ar[d]^{\phi_{i}} & \bbL_{\varphi_i}\ar[r]\ar@{=}[d] & \\
\bbL_{X_{i+1}}|_{X_i}\ar[r] & \bbL_{X_i}\ar[r] & \bbL_{\varphi_i}\ar[r] &   
}\eeq 
The relative cotangent complex $\bbL_{\varphi}=[\Omega_{x_{i+1}}|_{X_i}\to \Omega_{x_i}]$ is perfect of amplitude $[-1,0]$
and the lci pullback by $\varphi_i$ is the virtual pullback with perfect obstruction theory $\mathrm{id}:\bbL_{\varphi_i}\to \bbL_{\varphi_i}$. 

By Definition \ref{3.32}, we have the virtual fundamental class 
$$[X_i]\virt \in H_{r+d(x_i)}(X_i)$$
where $r$ is the rank of $E$. 
Moreover, by Theorem \ref{3.34}, we have the equality
$$[X_i]\virt =\varphi_i^![X_{i+1}]\virt$$
and thus  $\{[X_i]\virt\}$ is a class in the inverse limit $\varprojlim_i H_{r+d(x_i)}(X_i)=\cH_r(\cX)$. 
\begin{defi}\label{3.82}
Given a perfect obstruction theory for a \DM stack $\cX$ which is liftable to a good system $\{x_i:X_i\to \cX\}$ of approximations, the virtual fundamental class of $\cX$ is defined as
the limt \beq\label{3.83} [\cX]\virt=\varprojlim_{i} \, [X_i]\virt \in \cH_r(\cX).\eeq
\end{defi}
\begin{exam}\label{7.50}
Let $\cX=[X/G]$ be the quotient stack of a quasi-projective scheme acted on linearly by a linear algebraic group $G$. 
Then we have a good system $x_i:X_i=X\times_GU_i\to [X/G]=\cX$ of approximations by Example \ref{3.4}.
Since $U_i/G$ is a quasi-projective scheme, the projection 
$$\eta_i:X_i=X\times_G U_i\to U_i/G=:B_i$$ is a fibration with fiber $X$.  
Suppose $\phi:E_X\to \bbL_X$ is a perfect obstruction theory of $X$ which is a morphism in the $G$-equivariant derived category
of quasi-coherent sheaves on $X$. Then $\phi$ 
descends to a perfect obstruction theory $\phi:E\to \bbL_\cX$ and 
 lifts to a perfect obstruction theory
$$\phi_{\eta_i}:E_{\eta_i}\lra \bbL_{\eta_i}.$$
As $B_i$ is smooth, $\phi_{\eta_i}$ induces a perfect obstruction theory $\phi_i$ for $X_i$ by standard arguments \cite{BeFa}.
It is straightforward to see that $\phi$ and $\phi_i$ satisfy all the conditions in Definition \ref{3.80}. We thus
obtain the virtual fundamental class $[\cX]\virt\in \cH_r(\cX)$ of the quasi-projective \DM stack $\cX=[X/G]$. 
\end{exam}

Next we define the virtual pullbacks for limit intersection theories. 
To simplify the discussion, we consider only quasi-projective morphisms (cf. Remark \ref{3.87}). 
\begin{defi}\label{3.84} 
{Let $f:\cX\to \cY$ be a quasi-projective morphism of algebraic stacks in $\mathbf{St}_\bk^{ga}$ equipped with a perfect obstruction theory.}
Let $\{y_i:Y_i\to \cY, \varphi_i:Y_i\to Y_{i+1}\}$ be a good system of approximations and consider the fiber diagram
\beq\label{3.94}\xymatrix{
X_i\ar[rr]^{f_i}\ar[dd]_{x_i}\ar[dr] && Y_i\ar[dd]^(.6){y_i}\ar[dr]^{\varphi_i}\\
& X_{i+1}\ar[rr]^(.4){f_{i+1}}\ar[dl]^{x_{i+1}} && Y_{i+1}\ar[dl]^{y_{i+1}}\\
\cX\ar[rr]_f && \cY
}\eeq
so that $x_i:X_i=Y_i\times_\cY \cX\to \cX$ is a good system of approximations by Lemma \ref{2.78}.
The perfect obstruction theory for $f$ induces a perfect obstruction theory for $f_i$ and hence the virtual pullback $f_i^!$.  
By the commutativity of virtual pullbacks (Theorem \ref{3.34}), we have  
$$\varphi_i^!\circ f_{i+1}^!=f_i^!\circ \varphi_i^!$$
and thus $\{f_i^!\}$ define the \emph{virtual pullback} 
\beq\label{3.89} f^!=\varprojlim_{i}\, f_i^!: \cH_*(\cY)\lra \cH_{*+r}(\cX).\eeq
\end{defi}

\begin{theo}\label{3.90}
For quasi-projective morphisms in $\mathbf{St}_\bk^{ga}$, (1)-(3) in Theorem \ref{2.85} hold for the virtual pullback \eqref{3.89} of the limit intersection theory $\cH$. 
\end{theo} 
\begin{proof}
By Theorem \ref{3.34}, (1)-(3) of Theorem \ref{2.85} hold for the virtual pullbacks $f_i^!$ in $\mathbf{QSch}_\bk$. 
The theorem follows by taking limits.
\end{proof}

The following also follows directly from Theorem \ref{3.34}. 
\begin{prop}\label{3.91}
Let $f:\cX\to \cY$ be a quasi-projective morphism in  $\mathbf{St}_\bk^{ga}$ equipped with 
a perfect obstruction theory $\phi_f:E_f\to \bbL_f$. Suppose $\cY$ (resp. $\cX$) admits 
a perfect obstruction theory $\phi_\cY:E_\cY\to \bbL_\cY$ (resp. $\phi_\cX:E_\cX\to \bbL_\cX$) 
liftable to a good system $\{y_i:Y_i\to \cY\}$ of approximations (resp. the induced system 
$\{x_i:X_i=\cX\times_\cY Y_i\to \cX\}$). If $\phi_f$ and the induced perfect obstruction theories 
$\phi_{Y_i}:E_{Y_i}\to \bbL_{Y_i}$ and $\phi_{X_i}:E_{X_i}\to \bbL_{X_i}$  
fit into a commutative diagram of exact triangles
{\[\xymatrix{
E_{Y_{i}}|_{X_i}\ar[r]\ar[d]^{\phi_{Y_i}|_{X_i}}  & E_{X_i}\ar[r] \ar[d]^{\phi_{X_i}}& E_{f}|_{X_i}\ar[r]\ar[d]^{\phi_f|_{X_i}}&\\
\bbL_{Y_{i}}|_{X_i}\ar[r]  & \bbL_{X_{i}}\ar[r] & \bbL_{f_{i}}\ar[r]&\\
}\]}
we have $f^![\cY]\virt=[\cX]\virt.$
\end{prop}

A quasi-projective morphism of quasi-projective \DM stacks satisfies the assumptions in Theorem \ref{3.90} and Proposition \ref{3.91}
\begin{exam}
Let $\cX=[X/G]$ and $\cY=[Y/G]$ be quotients of quasi-projective schemes acted on linearly by a linear algebraic group $G$. 
Let $\hat{f}:X\to Y$ be a $G$-equivariant morphism that induces a morphism $f:\cX\to \cY$. 
A $G$-equivariant perfect obstruction theory $\hat{\phi}:\hat{E}\to \bbL_{\hat{f}}$ induces a perfect obstruction theory $\phi:E\to \bbL_f$. If $X$ and $Y$ are equipped with $G$-equivariant perfect obstruction theories that fit into a commutative diagram like \eqref{2.87} with $Z=\Spec\,\bk$, then Theorem \ref{3.90} and Proposition \ref{3.91} apply to $f$ and $\cX$, $\cY$.  
\end{exam}

\begin{rema}\label{3.87}
The virtual fundamental class \eqref{3.83} and the virtual pullback \eqref{3.89} can be defined in much more general setting. 
In fact, the virtual fundamental class $[\cX]\virt\in \cH_r(\cX)$ can be defined for any \DM stack with perfect obstruction theory. For the virtual pullback, we only need that $f:\cX\to \cY$ is of \DM type and $\cY$ admits a good system of approximations.
Here's the idea. 

Let $t:T\to \cX$ be a smooth morphism with $T$ a quasi-projective scheme and let $\imath:\fC_\cX\to \cE$ denote an embedding of the intrinsic normal cone of $\cX$ into a vector bundle stack. 
The composition $f_t=f\circ t:T\lra \cX$ 
may not admit a perfect obstruction theory, but we have a morphism
\beq\label{3.40} \fC_{T}\lra \fC_{\cX}|_T\mapright{\imath_T} \cE|_T\eeq
by \cite[2.20]{Man}, where the first arrow fits into the short exact sequence 
\beq\label{3.41} 0\lra \fN_{T/\cX}\lra \fC_{T}\mapright{h} \fC_{\cX}|_T\lra 0\eeq
of cone stacks by \cite[Remark 2.24]{Man}. As the last morphism $\fC_{T}\to \fC_{\cX}|_T$ is $\A^1$-equivariant smooth and surjective, the smooth pullback 
$$h^*:\cH_*(\fC_{\cX}) \lra \cH_{*+d(h)}(\fC_T)$$ 
is an isomorphism. Then we define
$$[\cX]\virt_t:=0^*_{t^*\cE}\circ (\imath_T)_*\circ (h^*)^{-1}\circ \mathrm{sp}_{T/\Spec \bk} (\mathbf{1})\in H_{r+d(t)}(T).$$
The virtual fundamental class of $\cX$ is defined as the limit
$$[\cX]\virt:=\varprojlim_{t:T\to \cX} \, [\cX]\virt_t\in \cH(\cX).$$
The virtual pullback is similarly defined. This general theory involves a lot of things to be checked and takes many pages. 
The details will appear in the second author's doctoral dissertation \cite{Park}. 
\end{rema}

\bigskip

\section{Cosection localization of virtual fundamental classes}\label{S9}

In \cite{KLc, KLk, KLq}, it was proved that the virtual fundamental class $[\cX]\virt$ for {Chow theory, K-theory} and Borel-Moore homology is localized to the zero locus of a cosection $\sigma$ when the obstruction sheaf $Ob_\cX=h^1(E^\vee)$ of the perfect obstruction theory $\phi:E\to \bbL_\cX$ admits a homomorphism $\sigma:Ob_\cX\to \sO_\cX$. This cosection localization turned out to be quite useful and led to many remarkable results (cf. \cite{CK, CLp, CLL, CLLL, Clad, FJR, GS1, GS2, HLQ, JiTh, KL1, KL2, KoTh, MPT, PT}).  

In this section, we generalize the cosection localization to arbitrary intersection theory on $\mathbf{QSch}_\bk$ and to limit intersection theories on the category $\mathbf{DM}_\bk^{ga}$ of \DM stacks in $\mathbf{St}^{ga}_\bk$. 

\medskip
Throughout this section, we fix an intersection theory $H_*$ for $\mathbf{QSch}_\bk$ (Definition \ref{2.1}) whose limit theory on $\mathbf{St}_\bk^{ga}$ is denoted by $\cH_*$ as before.  
By \eqref{4.30}, the projective pushforward gives us an isomorphism 
\beq\label{7.34}H_*(X_\redd)\cong H_*(X)\eeq for any quasi-projective scheme because all projective morphisms $f:Y\to X$ from smooth $Y$ factor through the reduced scheme $X_\redd$ of $X$. 
\begin{lemm}\label{4.31} For an algebraic stack $\cX$ with a good system $\{x_i:X_i\to \cX\}$ of approximations, the projective pushforward gives us an isomorphism 
$$\cH_*(\cX_\redd)\cong \cH_*(\cX).$$
\end{lemm}
\begin{proof}
Because $x_i$ is smooth, $X_i\times_\cX \cX_\redd=(X_i)_\redd$ and the induced smooth morphisms
$(X_i)_\redd\to \cX_\redd$ is a good system of approximations by Lemma \ref{2.78}.  
By taking the limit, the isomorphisms 
$H_*((X_i)_\redd)\cong H_*({X_i})$ for schemes induce the isomorphism in the lemma.
\end{proof}

\subsection{Cosection localization for quasi-projective schemes}\label{S9.1}

The goal of this subsection is to prove the following generalization of \cite[Theorem 1.1]{KLc} and \cite[Theorem 1.1]{KLk}. 
\begin{theo}\label{3.42}
Let $X$ be a quasi-projective scheme equipped with a perfect obstruction theory $\phi:E\to \bbL_X$ and a cosection $\sigma:Ob_X=h^1(E^\vee)\to \sO_X$ whose zero locus is denoted by $X(\sigma)$. For any intersection theory $H_*$ for $\mathbf{QSch}_\bk$, there is a cosection localized virtual fundamental class $$[X]_\loc\virt \in H_r(X(\sigma))$$ such that
$\imath_*[X]\virt_\loc=[X]\virt\in H_r(X)$ where $\imath$ denotes the inclusion of $X(\sigma)$ into $X$ and $r$ is the rank of $E$. Moreover $[X]\virt_\loc$ is deformation invariant. 
\end{theo}
Here $X(\sigma)$ is the closed subscheme defined by the ideal $\sigma(Ob_X)\subset \sO_X$. 

\begin{proof}
Since $X$ is quasi-projective, the dual $E^\vee$ of $E$ is represented by a two-term complex 
$[E_0\to E_1]$ of locally free sheaves on $X$, so that we have an exact sequence
$$E_0\lra E_1\lra h^1(E^\vee)=Ob_X\lra 0$$
and $\cE:=h^1/h^0(E^\vee)=[E_1/E_0]$. Let 
\beq\label{4.46} 
E_1(\sigma):=\ker (E_1\to Ob_X\mapright{\sigma} \sO_X), \quad \cE(\sigma):=[E_1(\sigma)/E_0].\eeq

By \cite[Proposition 4.3]{KLc}, the {intrinsic} normal cone $\fC_X$ has its support in $\cE(\sigma)$ 
and we have a commutative diagram
\beq\label{4.43} \xymatrix{
\fC_X\ar[r] & \cE\\
(\fC_X)_\redd \ar[r]\ar[u] & \cE(\sigma)\ar[u]
}\eeq
of closed immersions. Denoting the specialization homomorphism \eqref{4.21} for $X\to \Spec\bk$ by $\mathrm{sp}$, we have a commutative diagram 
\beq\label{4.44}\xymatrix{
H_*(\Spec \bk)\ar[r]^{\mathrm{sp}} \ar@{=}[d]& \cH_*(\fC_X)\ar[r]  & \cH_*(\cE)\ar[r]^{0^!_\cE} & H_{*+r}(X)\\
H_*(\Spec \bk)\ar[r]& \cH_*((\fC_X)_\redd)\ar[u]^\cong \ar[r] & \cH_*(\cE(\sigma))\ar[u] \ar@{.>}[r]^{0^!_{\cE,\loc}} & H_{*+r}(X(\sigma))\ar[u]_{\imath_*}
}\eeq
whose top row sends $\mathbf{1}$ to the virtual fundamental class $[X]\virt=[X]\virt_H$. 
All vertical arrows are projective pushforwards. 

Lemma \ref{4.45} below says that when $H_*=\Omega_*$ is the algebraic cobordism theory in \cite{LeMo}, there is a homomorphism $0^!_{\cE,\loc}$ which completes the last square. 
Hence for $H_*=\Omega_*$, we can define the \emph{cosection localized virtual fundamental class} $[X]\virt_{\Omega,\loc}$ as the image of $\mathbf{1}$ by the bottom row and the theorem is proved in this case. 

By Theorem \ref{2.2}, the algebraic cobordism theory is universal in the sense that for any intersection theory $H_*$ on $\mathbf{QSch}_\bk$ and a quasi-projective scheme $X$, there is a unique $\Omega_*(X)\to H_*(X)$ that preserves all the  operations in an intersection theory.  Hence the top row of  \eqref{4.44} fits into a commutative 
\beq\label{4.49}\xymatrix{
\Omega_*(\Spec \bk)\ar[r]^{\mathrm{sp}}\ar[d] \ar[d]& \Omega_*(\fC_X)\ar[r]\ar[d]  & \Omega_*(\cE)\ar[r]^{0^!_\cE} \ar[d]& \Omega_{*+r}(X)\ar[d]\\
H_*(\Spec \bk)\ar[r]^{\mathrm{sp}}& \cH_*(\fC_X)\ar[r] & \cH_*(\cE)\ar[r]^{0^!_{\cE}} & H_{*+r}(X)
}\eeq
and hence the image of the virtual fundamental class of $X$ in $\Omega_r(X)$ by the last vertical arrow is the virtual fundamental class $[X]\virt_H\in H_r(X)$. Then the image $[X]\virt_{H,\loc}\in H_r(X(\sigma))$ of $[X]\virt_{\Omega,\loc}\in \Omega_r(X(\sigma))$ by the universal homomorphism $\Omega_r(X(\sigma))\to H_r(X(\sigma))$ is the desired cosection localized virtual fundamental class for $H_*$. See Proposition \ref{4.59} for the deformation invariance. 
This completes our proof. 
\end{proof}

For the construction below, we will need the intersection with $-D$ for an effective Cartier divisor $D$. 
Let $F_H(u,v)\in H_*(\Spec\,\bk)[\![u,v]\!]$ be the formal group law \eqref{2.74} for $H_*$.
\begin{defi} \label{4.52} Let $\jmath:D\hookrightarrow X$ be an effective Cartier divisor on a quasi-projective scheme $X$ and let $L=\sO_X(D)$. 
By using the formal inverse $u\,g(u)$ in \eqref{4.51} of $u$, the \emph{intersection with $-D$}  is defined by
\beq\label{4.57}(-D)\cdot = \jmath^!\circ g(c_1(L)): H_*(X)\lra H_{*-1}(D).\eeq
\end{defi}
We will also use the following fact.
\begin{lemm} \label{4.55} \cite[Lemma 7.9]{Vish}
Let \[\xymatrix{ X\ar[r]^{\imath'} \ar[d]_{f'} & Y\ar[d]^f\\  
Z\ar[r]^\imath & W }\] 
be a Cartesian square of quasi-projective schemes where $\imath$ is a closed immersion and $f$ is projective. 
If $f$ is an isomorphism over $W-Z$, we have an exact sequence
$$\Omega_*(X)\mapright{(\imath'_*,-f'_*)} \Omega_*(Y)\oplus \Omega_*(Z)\mapright{f_*+\imath_*} \Omega_*(W)\lra 0.$$
\end{lemm}

\begin{lemm}\label{4.45} There is a homomorphism $0^!_{\cE,\loc}:\Omega_*(\cE(\sigma))\to \Omega_*(X(\sigma))$ that completes the last square of the commutative diagram \eqref{4.44} when $H_*=\Omega_*$ is the algebraic cobordism in \cite{LeMo}.
\end{lemm}
\begin{proof} 
By the definition of $0^!_\cE$, the last square in \eqref{4.44} is the commutative 
\beq\label{4.47}\xymatrix{
\Omega_*(\cE)\ar[r]^{\cong} & \Omega_{*+r_0}(E_1)\ar[r]^{0^!_{E_1}} & \Omega_{*+r_0-r_1}(X)\\
\Omega_*(\cE(\sigma))\ar[r]^{\cong} \ar[u]& \Omega_{{*+r_0}}(E_1(\sigma))\ar[u] \ar@{.>}[r]^{0^!_{E_1,\loc}} & \Omega_{{*+r_0-r_1}}(X(\sigma))\ar[u]_{\imath_*}
}\eeq
where $r_i$ is the rank of $E_i$ for $i=0,1$ so that $r=r_0-r_1$. 
Hence it suffices to complete the last square in \eqref{4.47}. 

Let $\rho:\tX\to X$ be the blowup of $X$ along $X(\sigma)$ and $\rho':D\to X(\sigma)$ be the restriction of $\rho$ to the exceptional divisor $D$. Then the cosection $$E_1\twoheadrightarrow Ob_X\mapright{\sigma} \sO_X$$ lifts to a surjective homomorphism 
$$\tE_1=\rho^*E_1\twoheadrightarrow \sO_{\tX}(-D)\subset \sO_{\tX}.$$
Let $E'$ be its kernel. 

Note that $E_1(\sigma)=E_1|_{X(\sigma)}\cup \ker(\sigma:E_1|_{X-X(\sigma)}\to \sO_{X-X(\sigma)})$. 
Applying Lemma \ref{4.45} to the fiber square
\[\xymatrix{
E'|_D\ar[r]^{\jmath'}\ar[d]_{\bar{\rho}} & E'\ar[d]^{\tilde{\rho}}\\
E_1|_{X(\sigma)}\ar[r]^\jmath & E_1(\sigma),
}\]
we obtain an exact sequence 
$$\Omega_*(E'|_D)\mapright{(\jmath'_*,-\bar{\rho}_*)} \Omega_*(E')\oplus \Omega_*(E_1|_{X(\sigma)})\mapright{\tilde{\rho}_*+\jmath_*} \Omega_*(E_1(\sigma))\lra 0.$$
In particular, for $\xi\in \Omega_*(E_1(\sigma))$, we can find a $\zeta\in \Omega_*(E')$ and $\eta\in \Omega_*(E_1|_{X(\sigma)})$ 
such that 
\beq\label{4.58}\xi=\tilde{\rho}_*\zeta+\jmath_*\eta.\eeq

Now we define  $0^!_{E_1,\sigma}:\Omega_*(E_1(\sigma))\lra \Omega_{*-r_1}(X(\sigma))$ by
\beq\label{4.48} 0^!_{E_1,\sigma}(\xi)=0^!_{E_1|_{X(\sigma)}}(\eta)+\rho'_*((- D)\cdot 0^!_{E'}\zeta)
\eeq
where $(-D)\cdot$ is \eqref{4.57}. 
For any other choice $(\zeta',\eta')$ for \eqref{4.58}, the difference $(\zeta-\zeta', \eta-\eta')$ equals
$(\jmath'_*\alpha, -\bar{\rho}_*\alpha)$ for some $\alpha\in \Omega_*(E'|_D)$. 
By a straightforward computation, we have $$\rho'_*((-D)\cdot 0^!_{E'}\jmath'_*\alpha)=\rho'_*(0^!_{\tilde{E}_1|_D}\alpha)=0^!_{E_1|_{X(\sigma)}}({\bar{\rho}_*\alpha}).$$ Hence, \eqref{4.58} is independent of the choice of $(\zeta,\eta)$.
\end{proof}

Finally we show that the map $0^!_{\cE,\loc}$ in Lemma \ref{4.45} is indepedent of the choice of the resolution $[E_0\to E_1]$ of the dual of the perfect obstruction theory $E$ and hence so is the virtual fundamental class $[X]\virt_\loc$. 
\begin{prop}\label{6.20}
The cosection-localized Gysin map $0^!_{\cE,\loc}$ is independent of the choice of the presentation $\fE \cong [E_1/E_0]$.
\end{prop}

\begin{proof}
Consider another presentation $[F_1/F_0] \cong \fE$. We may assume that there is a surjective morphism $\varphi : F_1 \to E_1$ making the diagram
$$\xymatrix{
& F_1 \ar[d]\ar[dl]_{\varphi}\\
E_1 \ar[r] &\fE
}$$
commute, up to a 2-isomorphism. Then $F_0=E_0\times_{E_1}F_1$. Consider the Cartesian square
$$\xymatrix{
F_1(\sigma) \ar[r] \ar[d]_{\psi} & F_1 \ar[d]^{\varphi}\\
E_1(\sigma) \ar[r] & E_1.
}$$
It remains to prove that
\beq\label{6.21}
0^!_{E_1,\loc} = 0^!_{F_1,\loc} \circ \psi^*.
\eeq
As in the proof of Lemma \ref{4.45}, let $\rho:\tX\to X$ be the blowup with exceptional divisor $D$ and let $\tE_1=\rho^*E_1$, $\tF_1=\rho^*F_1$, $\tilde{\varphi}=\rho^*\varphi$, $E'=\ker(\tE_1\twoheadrightarrow \sO_{\tX}(-D))$ and $F'=\ker(\tF_1\twoheadrightarrow \sO_{\tX}(-D))$. Then we have a commutative diagram of short exact sequences
$$\xymatrix{
0 \ar[r] & F' \ar[r] \ar[d]_{\mu} & \tF_1 \ar[r] \ar[d]^{\tilde{\varphi}} & \sO_{\tX}(-D) \ar[r] \ar[d]^{\id} & 0\\
0 \ar[r] & E' \ar[r] & \tE_1 \ar[r] & \sO_{\tX}(-D) \ar[r] & 0
}$$
over $\tX$. The proposition follows from 
$$ \rho'_* \circ (-D\cdot) \circ 0^!_{E'}  = \rho'_* \circ (-D\cdot) \circ 0^!_{F'} \circ \mu^*, \ \ 0^!_{E_1|_{X(\sigma)}} = 0^!_{F_1|_{X(\sigma)}} \circ (F_1|_{X(\sigma)} \mapright{\varphi} E_1|_{X(\sigma)})^*$$
which is a direct consequence of the functoriality of Gysin pullbacks.
\end{proof}

By the arguments for \cite[Theorem 2.6]{CKL}, we obtain a generalization of the virtual pullback formula \eqref{2.83}.
\begin{prop}\label{4.59}
Let $X$ (resp. $Y$) be a quasi-projective scheme equipped with a perfect obstruction theory $\phi_X:E_X\to \bbL_X$ (resp. $\phi_Y:E_Y\to \bbL_Y$). Suppose we further have a morphism $f:Y\to X$ equipped with a perfect obstruction theory $\phi_{Y/X}:E_{Y/X}\to \bbL_{Y/X}$ that fits into the commutative diagram 
\[\xymatrix{
E_X|_Y\ar[r]\ar[d]^{\phi_X} & E_Y \ar[r]\ar[d]^{\phi_Y} & E_{Y/X} \ar[r]\ar[d]^{\phi_{Y/X}}&\\
\bbL_X|_Y\ar[r] & \bbL_Y\ar[r] & \bbL_{Y/X} \ar[r]&
}\]
of exact triangles. If we have  a cosection $\sigma_X:Ob_X\to \sO_X$ which induces a cosection $\sigma_Y:Ob_Y\to Ob_X|_Y\to \sO_Y$, then the virtual pullback preserves the cosection localized virtual fundamental classes $$f^![X]\virt_\loc = [Y]\virt_\loc\in H_*(Y(\sigma_Y)).$$
\end{prop}
It is straightforward to adapt the proof of \cite[Theorem 2.6]{CKL} (see also  \cite[Theorem 5.2]{KLc}, \cite[Proposition 5.5]{KLk} and \cite{Kis}). 
We omit the proof. 

When $f:Y\to X$ is obtained as the pullback of a regular embedding $v$, Proposition \ref{4.59} tells us that 
$v^![X]\virt_\loc = [Y]\virt_\loc$ which implies that 
the cosection localized virtual fundamental class $[X]\virt\in H_r(X)$ is deformation invariant.

\medskip

\subsection{Cosection localization for limit intersection theories}\label{S9.2}
In this section, we generalize Theorem \ref{3.42} from intersection theories on quasi-projective schemes to limit intersection theories on  $\mathbf{St}_\bk^{ga}$. 

\begin{theo}\label{4.60}
Let $\cX$ be a \DM stack equipped with a good system $\{x_i:X_i\to \cX\}$ of approximations and a perfect obstruction theory $\phi:E\to \bbL_\cX$, liftable to the system $\{x_i\}$. 
If $\cX$ has a cosection $\sigma:Ob_\cX=h^1(E^\vee)\to \sO_\cX$, there is a cosection localized virtual fundamental class
\beq\label{4.63}[\cX]\virt_\loc=\varprojlim_i \,[X_i]\virt_\loc\in \cH_r(\cX(\sigma))\eeq
satisfying the deformation invariance and $\imath_*[\cX]\virt_\loc=[\cX]\virt$, where $\imath$ denotes the inclusion of the zero locus $\cX(\sigma)$ of the cosection $\sigma$ into $\cX$.  
\end{theo}
\begin{proof}
By assumption, we have a perfect obstruction theory $\phi_i:E_i\to \bbL_{X_i}$ for each $i$, that fit into diagrams 
\eqref{3.81}, \eqref{3.85} and \eqref{3.86}. 
The zero locus $\cX(\sigma)$ of $\sigma$ is defined as the closed substack defined by the ideal sheaf $I$ of $h^1(E^\vee)=:Ob_\cX$. Dualizing \eqref{3.81} for the smooth morphism $x_i:X_i\to \cX$, we obtain a commutative diagram
\[\xymatrix{
\bbT_{x_i}\ar[r]\ar@{=}[d] & \bbL^\vee_{X_i}\ar[d]\ar[r] & \bbL_\cX^\vee|_{X_i}\ar[d] \ar[r] &\\
\bbT_{x_i}\ar[r] & E_i^\vee\ar[r] & E^\vee|_{X_i}\ar[r] &
}\]
and an isomorphism $Ob_{X_i}=h^1(E^\vee_i)\cong h^1(E^\vee|_{X_i})$. Since tensor product is right exact, 
$h^1(E^\vee|_{X_i})\cong h^1(E^\vee)|_{X_i}=Ob_\cX|_{X_i}$ and thus we have a cosection
$$\sigma_i:Ob_{X_i}=h^1(E_i^\vee)\cong Ob_\cX|_{X_i}\lra \sO_{X_i}$$
whose image $I_i$ is $I\otimes_{\sO_\cX}\sO_{X_i}$ since $x_i$ is flat. Therefore we find that
\beq\label{4.61} X_i(\sigma_i)=\cX(\sigma)\times_\cX X_i, \quad X_i(\sigma_i)=X_{i+1}(\sigma_{i+1})\times_{X_{i+1}}X_i.\eeq
By Lemma \ref{2.78}, the restrictions $X_i(\sigma_i)\to \cX(\sigma)$ of $x_i$ form a good system of approximations. 

By Proposition \ref{4.59} and \eqref{3.86}, we find that 
$$[X_i]\virt_\loc =\varphi_i^![X_{i+1}]\virt_\loc.$$
Hence $\{[X_i]\virt_\loc\}\in \cH(\cX(\sigma))$ and \eqref{4.63} is well defined. 

The remaining properties follow from those of $[X_i]\virt_\loc$ for each $i$. 
\end{proof}

\begin{rema} As in Remark \ref{3.87}, Theorem \ref{4.60} holds in a much more general setting. The details will appear in \cite{Park}.
\end{rema}

\bigskip

\section{Torus localization of virtual fundamental classes}\label{S6}

In this section, we first prove the torus localization sequence in any intersection theory of quasi-projective schemes (Theorem \ref{7.1}), following the arguments in \cite{EG98}, and then we generalize the virtual torus localization formulas for {Chow theory} in \cite{GrPa} and {K-theory} in \cite{Qu} to any intersection theory of quasi-projective schemes (Theorem \ref{TLSc}) and to any limit intersection theory of stacks with good approximations (Theorem \ref{TLSt}), by the method of \cite{CKL}.

\subsection{Localization sequence for torus actions}\label{S6.1}
In this subsection, we generalize the torus localization in \cite{Brion, EG98} to any intersection theory of quasi-projective schemes.
The arguments are adapted from \cite{EG98}.  

Throughout this section, we fix an intersection theory $H_*$ for quasi-projective schemes. 
We denote the limit intersection theory (Definition \ref{2.5}) induced from $H_*$ by $\cH_*$.

Let $T$ be an $r$-dimensional split torus and $\widehat{T}=\Hom(T,\bG_m)$ be the character group. 
For any character $\lambda \in \widehat{T}$, denote by $\bk(\lambda)$ the corresponding 1-dimensional representation. 
Fix a $\ZZ$-basis $t_1,t_2,\cdots,t_r \in \widehat{T}\cong \ZZ^{\oplus r}$. 
Let 
$$V_i := \bk(t_1)^{\oplus i}\oplus \cdots \oplus \bk(t_r)^{\oplus i}\and 
U_i := (\bk(t_1)^{\oplus i}\setminus 0) \times \cdots \times (\bk(t_r)^{\oplus i}\setminus 0)$$
so that $U_i/T  = (\PP^{i-1})^{r}$. By the projective bundle formula \eqref{2.24}, letting $\zeta_i=c^T_1(\bk(t_i))$,
if $T$ acts trivially on a scheme $X$, we have
\beq\label{7.15}
H_*(X\times (\PP^{i-1})^r)=H_*(X)[\zeta_1, \cdots, \zeta_r ]/\langle \zeta_1^i, \cdots, \zeta_r^i \rangle ,\eeq
\[ H^T_d(X) 
= \prod_{n_1,n_2,\cdots,n_r\ge 0} H_{d+n_1+\cdots+n_r}(X) \cdot \zeta_1^{n_1} \cdots \zeta_r^{n_r}.
\]

For any quasi-projective scheme $X$ with a linear $T$-action, we define the \emph{equivariant intersection theory} by
\beq \label{6.50} H^T_d(X):=\varprojlim_i H_{d-r+\dim U_i}(X\times_TU_i)= \cH_{d-r}([X/T]).\eeq

Recall from \S\ref{S4.1} that for a $T$-equivariant morphism $f:X\to Y$ (or a morphism $[X/T]\to [Y/T]$ of stacks) and a $T$-equivariant vector bundle $E$ on $X$ (or a vector bundle $[E/T]\to [X/T]$), 
we have the projective pushforward $f_*:H^T_*(X) \to H^T_*(Y)$ when $f$ is projective, the smooth pullback $f^*:H^T_*(Y)\to H^T_{*+e}(X)$ when $f$ is smooth of relative dimension $e$,
the Chern classes $c_i(E):H^T_*(X)\to H^T_{*-i}(X)$, and the Gysin pullback $f^!:H^T_*(Y)\to H^T_{*-c}(X)$ when $f$ is a regular immersion of codimension $c$.





\begin{theo}\label{7.1}(Excision sequence)
Let $X$ be a quasi-projective scheme with a linear $T$-action.   
Let $\imath:X^T\hookrightarrow X$ be the inclusion of the fixed point locus and $\ell:X-X^T\hookrightarrow X$ be its complement. 
We have the exact sequence
$$0\lra H^T_*(X^T) \mapright{\imath_*}  H^T_*(X) \mapright{\ell^*} H^T_*(X- X^T) \lra 0 .$$
\end{theo}

\begin{proof}
By the excision property \eqref{2.19}, we have an exact sequence
$$ H_*(X^T\times_T U_i) \lra  H_*(X\times_T U_i) \lra H_*((X- X^T)\times_T U_i) \lra 0.$$
Fix $*=d$ and let $K_i$ be the kernel of the first arrow and $L_i$ be the kernel of the second arrow so that we have two exact sequences
$$0\lra K_i\lra H_{d+r(i-1)}(X^T\times_T U_i) \lra L_i\lra 0,$$ 
$$0\lra L_i\to H_{d+r(i-1)}(X\times_T U_i) \lra H_{d+r(i-1)}((X- X^T)\times_T U_i) \lra 0.$$
By \cite[Proposition 1.12.3]{KaSc}, it suffices to show that the inverse systems $\{K_i\}$ and $\{L_i\}$ satisfy the Mittag-Leffler condition \cite[Definition 1.12.1]{KaSc} and $\varprojlim_i K_i =0$. It is easy to see that these conditions hold once we prove 

\noindent  
(*) \emph{there is an integer $M>0$ such that the homomorphisms $K_i \to K_{i-M}$ are trivial for $i$ big enough.}

We may assume that $X^T$ is connected. Indeed, if $X^T=\bigsqcup_j X^T_j$, then
$$K_i \subseteq \bigoplus_s \ker ( H_{d+r(i-1)}(X^T_j \times_T U_i) \to H_{d+r(i-1)}((X- \sqcup_{j'\neq j}X^T_{j'} )\times_T U_i)   ).$$
Thus, it suffices to prove (*) for the case where the $T$-fixed locus is connected.

By our assumption that the $T$-action on $X$ admits a linearization, there is a smooth quasi-projective scheme $Y$ with a linear $T$-action and a $T$-equivariant embedding $X \hookrightarrow Y$. We then have the Cartesian square
$$\xymatrix{
X^T \times_T U_i \ar[r]^{\imath_i} \ar[d] & X \times_T U_i \ar[d]\\
Y^T \times_T U_i \ar[r]^{\jmath_i} & Y \times_T U_i
}$$
of closed immersions with $Y^T$ smooth (cf. \cite{Ive}). By taking a smaller $Y$ if necessary, we may assume that $Y^T$ is connected since $X^T$ is connected. By \eqref{6.51}, we have
$$\jmath_i ^! \circ {\imath_i}_* = c_s(N_i)$$
where $N_i$ is the normal bundle of  $\jmath_i$
and $s$ is the rank of $N_i$. Hence, we have
$$K_i=\ker({\imath_i}_*) \subseteq \ker (c_s(N_i))$$ and 
thus it suffices to show that the inverse system of $\{\ker (c_s(N_i))\}_i$ satisfies the condition (*) which follows from 
Lemma \ref{7.2} below, because the $T$-fixed part of $N_{Y^T/Y}$ is zero (cf. \cite{Ive}). 
\end{proof}

\begin{lemm}\label{7.2}
Let $X$ be a connected quasi-projective scheme with the trivial $T$-action and $N$ be a $T$-equivariant vector bundle of rank $s$ with nonzero weights only. 
Let $N_i=N\times_T U_i\to X\times U_i/T=X\times (\PP^{i-1})^r$ be the bundle induced by $N$ and $K_i$ denote the kernel of $$c_{s}(N_i) : H_{d+r(i-1)}(X \times (\PP^{i-1})^r) \lra H_{d+r(i-1)-s}(X \times (\PP^{i-1})^r).$$ 
Then there exists a positive integer $M$ such that the homomorphism 
$$k_{i-M,i}^* : K_i \lra K_{i-M}$$
which is the restriction of the Gysin pullback by the regular immersion $k_{i-M,i}: X\times (\PP^{i-M-1})^r\to X\times (\PP^{i-1})^r$
is trivial for $i$ big enough. 
\end{lemm}

\begin{proof}
We may decompose $N=\oplus_{j=1}^m N(\lambda_j)$ into subbundles where $T$ acts on $N(\lambda_j)$ with nonzero weight $\lambda_j=\sum_{l=1}^r a_{jl}t_l$ for $a_{jl}\in \ZZ$. Then 
$$N_i=N\times_T U_i\cong \bigoplus_{j=1}^m N(\lambda_j)\boxtimes \sO_{(\PP^{i-1})^r}(a_{j1}, \cdots, a_{jr})$$
as a bundle over $X\times (\PP^{i-1})^r.$ 
Let $\zeta_l=c_1(\sO_{(\PP^{i-1})^r}(e_l))$ for the standard basis vectors $e_l$ of $\ZZ^r$ so that
{$$c_1(\sO_{(\PP^{i-1})^r}(a_{j1}, \cdots, a_{jr}))=\sum_{l=1}^ra_{jl}\zeta_l+\text{(higher order terms)}.$$}
Let $s_j$ be the rank of $N(\lambda_j)$ so that $s=\sum_j s_j$.

By the projective bundle formula \eqref{7.15} and the Whitney sum formula, we may write the top Chern class of $N_i$ as 
\beq\label{7.35} c_s(N_i)=P+P', \quad {P=\prod_{j=1}^m c_1(\sO_{(\PP^{i-1})^r}(a_{j1}, \cdots, a_{jr}))^{s_j}}\eeq
where $P'$ is a polynomial of Chern classes of $N(\lambda_j)$ without constant term. 
It is easy to see that Chern classes are nilpotent as operators using the splitting principle 
since the first Chern classes are nilpotent as they decrease the dimensions of the supports of classes (cf. \cite[Theorem 2.3.13]{LeMo}).
Hence, 
there is a fixed integer $A$, which does not depend on $i$, such that $(P')^A=0$.

Let $\alpha\in K_i$. Then we have $P\alpha=-P'\alpha$ since $c_s(N_i)\alpha=0$.
By \eqref{7.15}, we may write
$$\alpha=\sum_{0\le n_1, \cdots, n_r\le i}\alpha_{n_1,\cdots, n_r}\zeta_1^{n_1}\cdots \zeta_r^{n_r},\quad \alpha_{n_1,\cdots, n_r}\in H_{d+n_1+\cdots+n_r}(X).$$
Since $P^{A}\alpha=0$, we have
$$\alpha_{n_1,\cdots, n_r}=0, \quad \text{if }n_l+{s\cdot A}<i \text{ for all }l.$$
Letting {$M=s\cdot A$}, we have the desired vanishing
$$k^*_{i-M,i}(\alpha)=\sum_{0\le n_1, \cdots, n_r\le i-M}\alpha_{n_1,\cdots, n_r}\zeta_1^{n_1}\cdots \zeta_r^{n_r}=0.$$
\end{proof}

Let $R=H^T_*(\Spec \bk)=H_*(\Spec \bk)[\![\zeta_1,\cdots,\zeta_r]\!]$. Let $\cQ \subseteq R$ be the multiplicative subset generated by the first Chern classes $\{c^T_1(\bk(\lambda)) \}_{\lambda \neq 0\in \widehat{T}}$.

\begin{prop}\label{7.3}
Let $X$ be a quasi-projective scheme with a linear $T$-action. If $X^T=\emptyset$, then there is a $q \in \cQ$ such that
$$q\cdot H^T_*(X) =0. $$
\end{prop}

\begin{proof}
If suffices to prove that there exists a $q \in \cQ$ such that
$$q \cdot H_{d+r(i-1)}(X\times_T U_i) = 0$$
for all $d$ and $i$. We will use an induction on the dimension of $X$.
By \eqref{4.30} and \eqref{7.34}, we may assume that $X$ is reduced. 
We may also assume that $X$ is irreducible because if $X=X_1 \cup X_2$, then $$H_{d+r(i-1)}(X_1\times_T U_i)\oplus H_{d+r(i-1)}(X_2\times_T U_i) \to H_{d+r(i-1)}(X\times_T U_i)$$
is surjective by \eqref{4.30} again.

By \cite[Proposition 4.10]{Thomason}, there is a nonempty $T$-invariant open subscheme $O \subseteq X$ and a quotient torus $\phi : T \to T'$ such that the $T$-action on $O$ induces an $T'$-action on $O$ and there is a principal bundle $O\to O/T'$. 
Since $X$ has no fixed points, $T'$ is not a trivial group. 
Shrinking $O$ if necessary, we may further assume that $O \to O/T'$ is a trivial $T'$-torsor. 
Hence, there is a $T$-equivariant map $\alpha : O \to T'$. 
A nontrivial character $\lambda \in \widehat{T'}$ induces a nontrivial character $\lambda \circ \phi \in \widehat{T}$. Then we have
$$c_1((O\times U_i \times \bk(\lambda \circ \phi))/T) \cdot H_{d+r(i-1)}(O\times_T U_i) = 0, \quad \forall i, d,$$
because the $T$-equivariant line bundle $O \times U_i \times \bk(\lambda \circ \phi) \to O \times U_i$ has a nowhere vanishing $T$-invariant section
$$s : O \times U_i  \to O \times U_i\times \bk(\lambda \circ \phi) : (x,y) \mapsto (x,y,\lambda\circ\alpha(x)\cdot v),\quad v\ne 0.$$
The proposition follows from the induction and the exact sequence
$$H_{d+r(i-1)}((X- O)\times_T U_i) \to H_{d+r(i-1)}(X\times_T U_i) \to H_{d+r(i-1)}(O\times_T U_i) \to 0$$
because the dimension of $X-O$ is smaller than the dimension of $X$.
\end{proof}

Localizaing the excision sequence in Theorem \ref{7.1} by $\cQ$, we obtain the following from Proposition \ref{7.3} because $H_*^T(X-X^T)\otimes_R R[\cQ^{-1}]=0.$
\begin{coro}\label{7.36} Under the assumptions of Theorem \ref{7.1}, we have 
$$\imath_*:H_*^T(X^T)\otimes_R R[\cQ^{-1}]\mapright{\cong} H_*^T(X)\otimes_R R[\cQ^{-1}].$$
\end{coro}

The Euler class of the normal bundle of the fixed locus $X^T$ is invertible. 
\begin{prop}\label{7.4}
Let $X$ be a quasi-projective scheme with the trivial $T$-action and $N$ be a $T$-equivariant vector bundle of rank $s$. 
If the fixed part of $N$ is zero, then the Euler class (top Chern class)
$$e(N) := c_s(N) : H^T_*(X)\otimes_R R[\cQ^{-1}] \to H^T_{*-s}(X)\otimes_R R[\cQ^{-1}]$$
is an isomorphism.
\end{prop}

\begin{proof}
From the proof of Lemma \ref{7.2}, we find that $c_s(N)$ is the sum of 
the product of a finite number of elements in $\cQ$ and a nilpotent operator. Hence $c_s(N)$ is invertible.
\end{proof}

\subsection{Virtual torus localization theorem for schemes}\label{S6.2}

The goal of this subsection is to prove the following by the method of \cite{CKL}. 

\begin{theo}\label{TLSc}
Let $X$ be a quasi-projective scheme equipped with a linear action of a torus $T$ and a $T$-equivariant perfect obstruction theory $\phi:E\to \bbL_X$. Let $H_*$ be an intersection theory on $\mathbf{QSch}_\bk$. 
Let $F=X^T$ denote the $T$-fixed subscheme and $\phi_F:E_F=E|_{F}^{\mathrm{fix}}\to \bbL_F$ be the perfect obstruction theory of $F$ (cf. \cite[Lemma 3.3]{CKL}). Let $N\virt=(E|_F^{\mathrm{mv}})^\vee$ be the virtual normal bundle of $F$. 
Then the virtual fundamental classes of $X$ and $F$ satisfy 
\beq\label{7.41}[X]\virt=\imath_*\frac{[F]\virt}{e(N\virt)} \in H_*^T(X)\otimes_R R[\cQ^{-1}].\eeq
\end{theo}
For the scheme structure of $F$ and the perfect obstruction theory $\phi_F$ etc, see \cite[\S3]{CKL} for instance. 
Fix a global resolution $N\virt=[N_0\to N_1]$ by locally free sheaves $N_0$ and $N_1$. 
The Euler class of $N\virt$ is $e(N\virt)=e(N_0)/e(N_1)$ which makes sense by Proposition \ref{7.4}.
It is easy to show that $e(N\virt)$ is independent of the global resolution.

\begin{proof}
By Corollary \ref{7.36}, there is a class $\alpha\in H_*^T(X^T)\otimes_R R[\cQ^{-1}]$ satisfying 
\beq\label{7.38}\imath_*\alpha=[X]\virt.\eeq

By \cite[(3.1)]{CKL}, the perfect obstruction theory $\phi$ for $X$ and the modified perfect obstruction theory $\bar{\phi}_F$ for $F$ fit into the commutative diagram of exact triangles
\[\xymatrix{
E|_F\ar[r]\ar[d]_{\phi} & \bar{E}_F\ar[d]^{\bar{\phi}_F}\ar[r] & N_0^\vee[1]\ar[d]\ar[r] &\\
\bbL_X|_F\ar[r] & \bbL_F\ar[r] & \bbL_{F/X}\ar[r] &
}\]
where $\bar{E}_F$ is defined by the exact triangle
$\bar{E}_F\to E_F\to N_1^\vee[2]$ (cf. \cite[(3.2)]{CKL}). 
By Theorem \ref{3.90} (2), we have 
\beq\label{7.39} \imath^![X]\virt =[F]\virt\cap e(N_1)\eeq
by applying the proof of \cite[Lemma 3.7]{CKL} (with $\sigma=0$) where $\imath^!$ was taken 
with respect to the relative perfect obstruction theory $N_0^\vee[1]$. 
By \eqref{7.38} and \eqref{7.39}, we have 
\beq\label{7.40} e(N_0)\cap \alpha=\imath^!\imath_*\alpha=[F]\virt\cap e(N_1).\eeq
Therefore we have
$$[X]\virt=\imath_*\alpha=\imath_*\frac{[F]\virt}{e(N\virt)}.$$

\end{proof}

\begin{rema}
(1) With the same proof, one can prove the torus localization formula for the cosection localized virtual fundamental classes which generalizes \cite[Theorem 3.5]{CKL}.

(2) Theorem \ref{TLSc} generalizes the virtual torus localization in \cite{GrPa} for {Chow theory}. See \cite{Qu} for the proof of \eqref{7.41} in K-theory.
\end{rema}

\subsection{Virtual torus localization theorem for stacks}\label{S6.3}

Let $G$ be a linear algebraic group and $T$ be a torus. 
Let $X$ be a quasi-projective scheme with a linear action of $G\times T$. 
Let $\phi:E\to \bbL_X$ be a $G\times T$-equivariant perfect obstruction theory of $X$. 
Suppose $G$ acts on $X$ with only finite stabilizers so that $\cX=[X/G]$ is a \DM stack
and $\phi$ descends to a perfect obstruction theory of $\cX$.  
Let $F=X^T$ be the $T$-fixed locus and $N\virt$ be the virtual normal bundle in Theorem \ref{TLSc}.
By Example \ref{7.50}, we have induced perfect obstruction theories on 
the good approximations $X_i=X\times_G U_i$. 
By Theorem \ref{3.90}, we have the virtual fundamental classes
$$[\cX]\virt\in \cH_*^T(\cX)=\cH_*([X/G\times T]), \quad [\cF]\virt\in \cH_*^T(\cF)=\cH_*([F/G\times T])$$
defined as the limits of $[X_i]\virt$ and $[F_i]\virt$ where $\cF=[F/G]$ and $F_i=F\times_G U_i$. 

Now for each $i$, by Theorem \ref{TLSc}, we have the torus localization formula,
$$[X_i]\virt=\imath_*\frac{[F_i]\virt}{e(N_i\virt)}$$
where $N_i\virt$ is the virtual normal bundle of $F_i$ in $X_i$. 
Since the virtual pullback commutes with projective pushforwards and Chern classes, upon taking the limit over $i$, we obtain
the following.

\begin{theo}\label{TLSt}
Under the above assumptions, we have 
\beq\label{7.51}
{[\cX]\virt=\imath_*\frac{[\cF]\virt}{e(\cN\virt)} \in \cH^T_*(\cX)\widehat{\otimes}_R R[\cQ^{-1}] := \varprojlim_i H^T_*(X_i)\otimes_R R[\cQ^{-1}]}
\eeq
where $e(\cN\virt)$ is defined as the limit of $e(N_i\virt)$.
\end{theo}

\bibliographystyle{amsplain}

\end{document}